\documentclass[a4paper,11pt]{article}
\usepackage[english]{babel}

\usepackage{amssymb}
\usepackage{amsmath}
\usepackage{amsthm}
\usepackage{color}

\usepackage{mathabx}

\usepackage{hyperref}

\newtheorem{prop}{Proposition}[section]
\newtheorem{define}{Definition}[section]
\newtheorem{lemma}[prop]{Lemma}
\newtheorem{theo}{Theorem}[section]
\newtheorem{coro}{Corollary}[section]

\numberwithin{equation}{section}

\theoremstyle{remark}
\newtheorem{rmq}[prop]{Remark}

\newcommand{\mL}{\mathcal{L}}

\newcommand{\C}{\mathbb{C}}
\newcommand{\R}{\mathbb{R}}
\newcommand{\N}{\mathbb{N}}
\newcommand{\Z}{\mathbb{Z}}

\title{Global Strichartz estimates for the Schr\"odinger equation 
with non zero boundary conditions and applications}
\author{Corentin Audiard
\footnote{Sorbonne Universit\'es, UPMC Univ Paris 06, UMR 7598, Laboratoire Jacques-Louis Lions, 
F-75005, Paris, France }
\footnote{CNRS, UMR 7598, Laboratoire Jacques-Louis Lions, F-75005, Paris, France}}
\begin{document}
\maketitle

\begin{abstract}
We consider the Schr\"odinger equation on a half space in any dimension with a class of nonhomogeneous boundary 
conditions including Dirichlet, Neuman and the so-called transparent boundary conditions. 
Building upon recent 
% a recent argument due to Y.Ran, S.M. Sun and B.Y. Zhang, 
local in time Strichartz estimates (for Dirichlet boundary conditions), we obtain global 
Strichartz estimates for initial data in $H^s,\ 0\leq s\leq 2$ and boundary data in a natural 
space $\mathcal{H}^s$. For $s\geq 1/2$, the issue of compatibility conditions requires a thorough analysis 
of the $\mathcal{H}^s$ space. As an application we solve nonlinear Schr\"odinger equations and
construct global asymptotically linear solutions for small data. 
A discussion is included on the appropriate notion of scattering in this framework, and the optimality of the 
$\mathcal{H}^s$ space.
\end{abstract}
\begin{abstract}
On consid\`ere l'\'equation de Schr\"odinger sur le demi espace en dimension arbitraire pour une classe 
de conditions au bord non homog\`enes, incluant les conditions de Dirichlet, Neumann, et ``transparentes''.
Le principal r\'esultat consiste en des estimations de Strichartz globales pour des donn\'ees initiales 
$H^s$, $0\leq s\leq 2$ et des donn\'ees au bord dans un espace naturel $\mathcal{H}^s$, il 
am\'eliore les estim\'ees de Strichartz locales en temps obtenues r\'ecemment par 
d'autres auteurs dans le cas des conditions de Dirichlet. Pour $s\geq 1/2$, la d\'efinition des 
conditions de compatibilit\'e requiert une \'etude pr\'ecise des espaces $\mathcal{H}^s$. En application, on r\'esout 
des \'equations de Schr\"odinger non lin\'eaires, et on construit des solutions dispersives globales si les donn\'ees 
sont petites. On discute \'egalement le sens pr\'ecis donn\'e \`a ``solution dispersive'', ainsi que la question 
de l'optimalit\'e de l'espace $\mathcal{H}^s$.
\end{abstract}

\tableofcontents
\noindent
\section{Introduction}
We consider the initial boundary value problem (IBVP) for the Schr\"odinger equation on a half space
\begin{equation}\label{IBVP}
 \left\{
 \begin{array}{ll}
  i\partial_tu+\Delta u=f,\\
  u|_{t=0}=u_0,\\
  B(u|_{y=0},\partial_{y}u|_{y=0})=g,
 \end{array}
\right.
(x,y,t)\in \R^{d-1}\times \R^+\times \R_t^+,
\end{equation}
where the notation $\R_t$ emphasizes the time variable. $B$ is defined as follows: 
we denote $\mathcal{L}$ the Fourier-Laplace transform on $\R^{d-1}\times \R_t^+$
\begin{equation*}
g\to \mathcal{L}g(\xi,\tau):=\int_0^\infty \int_{\R^{d-1}}e^{-\tau t-ix\xi}g(x,t)dxdt,\ (\xi,\tau)\in \R^{d-1}\times 
\{z\in \C:\ \text{Re}(z)\geq 0\},
\end{equation*}
and $B$ satisfies 
\begin{eqnarray*}
\mathcal{L}(B(a,b))=b_1(\xi,\tau)\mathcal{L}(a)+b_2(\xi,\tau)\mathcal{L}(b),
\text{with }b_1,b_2 \text{ smooth on $\text{Re}(\tau)>0$ and }\\
\forall\,\lambda>0,\ b_1(\lambda\xi,\lambda^2\tau)=b_1(\xi,\tau),\
b_2(\lambda\xi,\lambda^2\tau)=\lambda^{-1}b_2(\xi,\tau).
\end{eqnarray*}
% More precise assumptions on $b_1$ and $b_2$ are given in section \ref{linestim}.
This kind of boundary conditions was considered by the author \cite{Audiard2} for a large class 
of dispersive equations on the half space. They are natural considering the homogeneity of the equation, they include  
Dirichlet ($b_1=1,\ b_2=0$) and Neuman boundary conditions 
($b_1=0,\ b_2=(|\xi|^2-i\tau)^{-1/2}$, see section \ref{linestim} for 
the choice of the square root), but 
also the important case of \emph{transparent} boundary conditions 
($b_1=1,\ b_2=-(|\xi|^2-i\tau)^{-1/2}$). The label transparent comes from the fact 
that the solution of the homogeneous IBVP with transparent boundary conditions coincides on $y\geq 0$ with the 
solution of the Cauchy problem that has for initial value the function $u_0$ extended by $0$ for $y\leq 0$ 
(see \cite{Besse4}).
\\
Our aim here is to prove the well-posedness of the IBVP under natural assumptions on $B$ detailed in section 
\ref{linestim}, and prove that the solutions satisfy so-called Strichartz estimates.\\
% The theory of the IBVP for linear and nonlinear Schr\"odinger equations is less advanced than the one 
% for the Cauchy problem, but significant progresses have been made over recent years. In 
% particular the existence of Strichartz estimates has been a very active topic for the homogeneous Dirichlet 
% boundary value problem in non trivial domains. 
Let us recall that the linear, pure Cauchy problem on $\R^d$ can be solved by elementary semi-group arguments, and 
Strichartz estimates are a key tool for the 
the analysis of nonlinear Schr\"odinger equations (NLS) (see the reference book \cite{Cazenave}). They can be seen as 
a consequence of the dispersion estimate $\|e^{it\Delta}u_0\|_{L^\infty}
\lesssim \|u_0\|_{L^1}/t^{d/2}$, and they read for $p>2$
\begin{equation}\label{strichartzCauchy}
\text{ for }p>2,\ \frac{2}{p}+\frac{d}{q}=\frac{d}{2},\ \|e^{it\Delta}u_0\|_{L^p(\R_t,L^q(\R^d)}\lesssim 
\|u_0\|_{L^2(\R^d)} .
\end{equation}
Any pair $(p,q)$ that satisfies the identity above is called admissible. 
In the limit case $p^*=2, q^*=2d/(d-2)$, in view of the critical Sobolev embedding $H^1\hookrightarrow L^{q*}$ 
such estimates correspond (scaling wise) to a gain of one derivative. 
% , and are a key tool for the well-posedness of semi-linear 
% Schr\"odinger equations .
It is easily seen that \eqref{strichartzCauchy} remains true if $\R_t$ is replaced by $[0,T]$, and by H\"older's inequality, 
the estimate is true on $[0,T]$ for $q\geq 2,\ 2/p+d/q\geq d/2$. For such indices it is usually called 
a Strichartz estimate with ``loss of derivatives''. \\
The study of the IBVP is significantly more difficult even for homogeneous Dirichlet boundary conditions: 
to our knowledge the existence of dispersion 
estimates is still an open problem, and it is now well understood that Strichartz
estimates strongly depend on the geometry of the domain. 
One of the first breakthroughs for the analysis of the homogeneous BVP was due to Burq, G\'erard and Tzvetkov 
\cite{BGT}, who proved that if the domain is \emph{non trapping}\footnote{A typical example is the exterior of a compact 
star shaped domain.} and $\Delta_D$ is the Dirichlet Laplacian
\begin{equation*}
\text{for }p\geq 2,\ \frac{1}{p}+\frac{d}{q}=\frac{d}{2},\ \|e^{it\Delta_D}u_0\|_{L^pL^q} \lesssim \|u_0\|_{L^2},
\end{equation*}
this corresponds to Strichartz estimates with loss of $1/2$ derivative.
Numerous improvements have been obtained since \cite{Anton2}\cite{BSS}, up to Strichartz estimates without 
loss of derivatives \cite{Ivanoviciconv}\cite{BSS}, and their usual consequences for semilinear problems. 
Very recently, Killip, Visan and Zhang \cite{KVZ} shrinked even more the gap by proving the global 
well-posedness of the quintic defocusing Schr\"odinger equation in dimension 3, while the same result 
for the Cauchy problem (see \cite{CKSTT}, 2008) is a recent (and spectacular) achievement.\\
Less results are available for nonhomogeneous boundary value problems. Actually, even in the simplest settings of 
a half space the two following fundamental questions have not received completely satisfying answers yet
\begin{enumerate}
\item Given smooth boundary data, what algebraic condition should satisfy $B$ for the BVP to be well-posed ?
\item For such $B$, given $s\geq 0$ what is the optimal regularity of 
the boundary data to ensure $u\in C_tH^s$?
\end{enumerate}
In dimension one, with Dirichlet boundary conditions, question $2$ is now well understood : 
for a solution $u \in C_tH^s(\R^+)$, the natural  space for the boundary 
data is $H^{s/2+1+4}(\R_t^+)$, see the work of  Holmer \cite{Holmer} and Bona, Sun and Zhang 
\cite{BSZ}, which includes interesting discussions on the optimality of $H^{s/2+1/4}(\R_t^+)$ . 
An easy way to understand this regularity assumption is that it is precisely the regularity of the trace of solutions 
of the Cauchy problem, as can be seen of the celebrated sharp Kato smoothing. Let us recall here the classical 
argument of \cite{KPV}
\begin{eqnarray*}
e^{it\Delta}u_0=\int_{\R}e^{-it|\xi|^2}e^{ix\xi}\widehat{u_0}d\xi
=\frac{1}{2\pi}\int_{\R^+}e^{-it\eta}\big(e^{ix\sqrt{\eta}}\widehat{u_0}
+e^{i-x\sqrt{\eta}}\widehat{u_0}\big)d\xi\\
\Rightarrow \|e^{it\Delta}u_0|_{x=0}\|_{\dot{H}^{s/2+1/4}}
\sim \int_{\R^+}(|\widehat{u_0}(\sqrt {\eta})|^2+|\widehat{u_0}(-\sqrt {\eta})|^2)
|\eta|^{s+1/2}d\eta 
\sim \int_{\R}|\widehat{u_0}(\xi)|^2  |\xi|^{2s}d\xi\\
\leq\|u_0\|_{H^s}^2.
\end{eqnarray*}
Moreover \cite{Holmer}, \cite{BSZ} derived Strichartz estimates  without loss of derivatives, so that local 
well-posedness can be deduced for various nonlinear problems. Global well-posednes 
results are also available in \cite{BSZ} for slightly smoother boundary data, precisely $H^{s/2+1/2}(\R_t^+)$. \\
The BVP in dimension $\geq 2$ is significantly more difficult, because the geometry can be more complex, and 
waves propagating along the boundary are harder to control (this issue appears even with the trivial geometry of 
the half space). We expect that the answer to question $2$ strongly depends on the domain. Due to its role for control 
problems, the Schr\"odinger equation in bounded domain has received significant attention, see 
\cite{Lasiecka2,Ozsari,Rosier} and references therein.  In unbounded domains with non trivial geometry, the regularity 
of the boundary data is different and Strichartz estimates with loss can be derived (see the author's 
contribution \cite{Audiard6}).\\
Let us focus now on the case where the domain is the half space. The Schr\"odinger equation shares some (limited) 
similarities with hyperbolic equations, for which question 1 has 
been clarified in the seminal work of Kreiss \cite{Kreiss}: there is a purely algebraic condition, the so-called 
Kreiss-Lopatinskii condition, which leads to Hadamard type instability if it is violated (see the book 
\cite{Benzoni3} section 4 and references therein). This condition was extended by the author in 
\cite{Audiard2} for a class of linear dispersive equations posed on the half space. 
A consequence of the main result was that if this condition is satisfied then \eqref{IBVP} is well posed 
in $C_tH^s$ for boundary data in $L^2(\R_t,H^{s+1/2}(\R^{d-1}))\cap 
H^{s/2+1/4}(\R_t,L^2)$ (a space that, scaling wise, is a natural higher 
dimensional version of $H^{s/2+1/4}(\R_t)$). We point out however that our (uniform) Kreiss-Lopatinskii 
condition derived was quite restrictive, and in particular forbid the Neuman boundary condition, a limitation which
is lifted here.
\\
% Actually, the boundary data space $L^2(\R_t,H^{s+1/2}(\R^{d-1})\cap H^{s/2+1/4}(\R_t,L^2)$ can be slightly improved.
On the issue of Strichartz estimates, Y.Ran, S.M.Sun and B.Y.Zhang considered in \cite{RSZ} the 
IBVP \eqref{IBVP} on a half space with nonhomogeneous Dirichlet boundary conditions. They derived  
explicit solution formulas in the spirit of their work on the Korteweg de Vries 
equation with J.Bona \cite{BSZ}, and managed to use them to obtain local in time 
Strichartz estimates without loss of derivatives. A very interesting feature was that the existence of solutions 
in $C_TH^s$ only required boundary data in some space $\mathcal{H}^s$ which has the same scaling as 
$L^2_tH^{s+1/2}\cap H^{s/2+1/4}_tL^2$ but is slightly weaker. We refer to paragraph \ref{secHs} 
for a precise definition of $\mathcal{H}^s$. 
The space $\mathcal{H}^s$ is in some way optimal, as it is exactly the 
space where traces of solutions of the Cauchy problem belong, see proposition \ref{estimcauchy}. Note however that in 
the appendix we provide a construction showing that it is less accurate for evanescent waves (solutions 
that exist only for BVPs and remain localized near the boundary).\\
Although not stated explicitly in \cite{RSZ}, we might roughly summarize 
their linear results as follows: 
\begin{theo}[\cite{RSZ}]\label{theoRSZ}
For $s\geq 0$, $s\notequiv 1/2 [2\Z]$, $(u_0,f,g)\in H^s(\R^{d- 1}\times \R^+)\times 
L^{1}([0,T],H^s)\times \mathcal{H}^s([0,T])$. If $(u_0,f,g)$ satisfy appropriate compatiblity conditions, 
the IBVP \eqref{IBVP} with Dirichlet boundary conditions 
has a unique solution $u\in C([0,T],H^s)$, moreover for any $(p,q)$
such that $p>2,\ \frac{2}{p}+\frac{d}{q}=\frac{d}{2}$ and $T>0$
it satisfies the a priori estimate 
\begin{equation*}
\|u\|_{L^p([0,T],W^{s,q})}\lesssim \|u_0\|_{H^s}+
\|f\|_{L^1([0,T],H^s)}+\|g\|_{\mathcal{H}^s([0,T])}. 
\end{equation*}
\end{theo}
% \noindent Note that for applications to nonlinear problems, the authors chose to work in the more convenient space 
% $H^{s/2+1/4}_tL^2\cap L^2_tH^{s+1/2}$ (however since the paper is a preprint we are not aware if some more recent 
% version handles the weaker $\mathcal{H}^s$ space). 
\noindent Our two main improvements are that 
% we work only in $\mathcal{H}^s$ spaces, 
we allow more general boundary conditions, and our Strichartz estimates are \emph{global} in time with a larger 
range of integrability indices for $f$ (any dual admissible pair).\\
% We adress questions 1 and 2 for the Schr\"odinger equation posed on a half space in various settings. 
% We first treat the pure BVP ($f=0,\ u_0=0$) for a general class of boundary conditions. 
For the full IBVP the smoothness 
of solutions does not only depend on the smoothness of the data, but also on some \emph{compatibility conditions}, the 
simplest one being $u_0|_{y=0}=g|_{t=0}$ in the case of Dirichlet boundary conditions.
% 
% , and we handle the case $s=1/2$ by introducing a new ``global'' compatibility 
% condition between $u_0$ and $g$, see equation \eqref{compaglob}
This compatibility condition 
is trivially satisfied if $u_0|_{y=0}=g|_{t=0}=0$ (that is, $u_0\in H^1_0$), but the non trivial case is mathematically 
relevant and important for nonlinear problems. It is delicate 
to describe compatibility conditions for a general boundary operator $B$, therefore we shall
split the analysis in the following two simpler problems :
\begin{itemize}
 \item General boundary conditions, ``trivial'' compatibility conditions in theorem \ref{thtrivial},
 \item Dirichlet boundary conditions, general compatibility conditions in theorem \ref{maintheo}.
\end{itemize}
As $\mathcal{H}^s$ is not embedded into continuous functions, even $g|_{t=0}$ does not have an immediate meaning. 
Therefore we thoroughly study the functional spaces $\mathcal{H}^s$ in paragraph \ref{secHs}, including trace properties 
which allow us to rigorously define the compatibility conditions, including the intricate case $s=1/2$ where $g|_{t=0}$ 
has no sense, but a global compatiblity condition is required. The main new consequence for nonlinear 
problems is a scattering result in $H^1$ for $(u_0,g)$ small in $H^1\times \mathcal{H}^1$.
All previous global well-posedness results require more smoothness on $g$.
\paragraph{Statement of the main results}
Let us begin with a word on the compatiblity condition $u_0|_{y=0}=g$: if 
$u_0\in H^s(\R^{d-1}\times \R^+)$, $s>1/2$, $u_0|_{y=0}$ is 
well defined and belongs to $H^{s-1/2}(\R^{d-1})$, moreover it is proved in 
proposition \ref{structureH} that if $g\in \mathcal{H}^s$ then $g|_{t=0}\in 
H^{s-1/2}(\R^{d-1})$, therefore if $u\in C_tH^s$ solves \eqref{IBVP}, necessarily 
\begin{equation}\label{compa}
\text{for }s>1/2,\ g|_{t=0}=u_0|_{y=0}.
\end{equation}
\eqref{compa} is the first order compatibility condition. 
If $s=1/2$, \eqref{compa} does not makes sense, but a subtler condition is required: let $\Delta'$ the laplacian on 
$\R^{d-1}$, then
\begin{equation}\label{compaglob}
\text{if }s=1/2,\ \int_{\R^{d-1}}\int_0^\infty \frac{|e^{-it^2\Delta'}g(x,t^2)-u_0(x,t)|^2}{t}dtdx<\infty.
\end{equation}
This is reminiscent of the famous Lions-Magenes global compatibility condition for traces on domains with corners, 
with a twist due to the Schr\"odinger evolution, see definition \eqref{def00} and paragraph \ref{mainproof} for 
more details.
When we say "the compatibility condition is satisfied",
we implicitly mean the strongest compatiblity condition that makes sense so that for $s<1/2$ nothing is required.
It is not difficult to define recursively higher order compatibility conditions (see e.g. \cite{Audiard6} section 2). 
Note however that higher order compatibility conditions involve also the trace $f|_{y=t=0}$, which makes sense only if 
$f$ has some time regularity. We do not treat this issue in this paper.\\
For nonlinear applications we are only interested by the $H^1$ regularity, so we choose
to consider indices of regularity $s\in [0,2]$. 
Our main result requires a few notions : see section \ref{funcanalysis} for the 
definition of the functional spaces $\mathcal{H}^s$, $\mathcal{H}^s_0$ and $\mathcal{H}^{1/2}_{00}$
and section \ref{linestim} for the definition of the 
Kreiss-Lopatinskii condition.
\\
We use the following definition of solution:
\begin{define}\label{defsol}
A function $u\in C(\R_t^+,L^2)$ is a solution of \eqref{IBVP} if there exists a  sequence 
$(u_0^n,f^n,g^n)\in H^2(\R^{d-1}\times \R^+)\times L^p(\R_t^+,W^{2,q})\times (L^2(\R^+_t,H^2)\cap H^1(\R^+_t,L^2))$, 
with 
\begin{equation*}
\|(u_0,f,g)-(u_0^n,f^n,g^n)\|_{L^2\times L^{p'_1}_tL^{q'_1}\times \mathcal{H}^0}\longrightarrow_n 0,
\end{equation*}
such that there exists a solution $u^n\in C_tH^2\cap C^1_tL^2$ to the corresponding IBVP and $u_n$ converges 
to $u$ in $C_tL^2$.
A $C_tH^s$ solution is a solution in the $C_tL^2$ sense with additional regularity.
% For general boundary conditions, if $(u_0,f,g)\in L^2\times \big(L^{p'_1}L^{q'_1} \cap L^{p'_1}L^{q'_1}\big)
% \times L^2$, $s>3/2$, a solution to \eqref{IBVP} is a function 
% $u\in C_tH^2$ which is a distributional solution and satisfy the initial and boundary condition.\\
% For $s\leq 3/2$, $u\in C_tH^s$ is a solution to \eqref{IBVP} if there exists a sequence 
% $(u_0^n,f^n,g^n)\in H^2_0\times \big(L^{p'_1}B^2_{q'_1,2} 
% \cap B^{1}_{p'_1,2}L^{q'_1}\big)\times \mathcal{H}^2_0$, with 
% \begin{equation*}
% \|(u_0,f,g)-(u_0^n,f^n,g^n)\|_{H^s_0\times \big(L^{p'_1}B^s_{q'_1,2} 
% \cap B^{s/2}_{p'_1,2}L^{q'_1}\big)\times \mathcal{H}^s_0}\to_n 0
% \end{equation*}
% such that the corresponding solution $u^n$ converges to $u$ in $C_tH^s$.\\
% In the case of Dirichlet boundary conditions, for $s<1/2$, $u\in C_tH^s$ is a solution if there exists a sequence 
% $(u_0^n,f^n,g^n)\in H^2_0\times \big(L^{p'_1}B^2_{q'_1,2} 
% \cap B^{1}_{p'_1,2}L^{q'_1}\big)\times \mathcal{H}^2_0$, with
% \begin{equation*}
% \|(u_0,f,g)-(u_0^n,f^n,g^n)\|_{H^s_0\times \big(L^{p'_1}B^s_{q'_1,2} 
% \cap B^{s/2}_{p'_1,2}L^{q'_1}\big)\times \mathcal{H}^s_0}\to_n 0
% \end{equation*}
% \emph{and that satisfy the compatibility condition}
% such that $u^n$ converges to $u$ in $C_tH^s$.\\
% For $s=1/2$, a function $u\in C_tL^2$ is a solution if it is a solution in the $C_tL^2$ sense.
\end{define}
% Of course, the definition makes sense only if such sequences exist. This is a simple density argument 
% in the case of general boundary conditions, for Dirichlet conditions the existence of approximating sequences 
% that satisfy the compatibility condition follows from an argument in the classical paper \cite{RauchMassey}. 
% It can be shortly summarized as follows: take a sequence $g^n\in \mathcal{H}^2$ that converges to $g$, 
% $g^n|_{t=0}\in H^{3/2}$ (see proposition \ref{structureH} later), for a lifting operator 
% $R:\ H^{3/2}(\R^{d-1})\to H^{2}(\R^{d-1}\times \R^+)$, 
% it is easily seen that 
% $\|(Rg^n|_{t=0})(x,\lambda y)\|_{H^{s}}\to 0$ as $\lambda\to \infty$ if $s<1/2$, therefore if $\lambda_n$ is well chosen 
In our statements we shall use the following convention for any $(p,q)\in [1,\infty]^2$
\begin{equation}
B^0_{q,2}(\R^{d-1}\times \R^+):=L^q,\ B^2_{q,2}(\R^{d-1}\times \R^+):=W^{2,q},\ B^0_{p,2}(\R_t^+)=L^p,\ 
B^1_{p,2}(\R_t^+)=W^{1,p}.
\end{equation}
These equalities are \emph{not true} for the usual definition of Besov spaces, but they allow us 
to give shorter statements for a regularity parameter $s\in [0,2]$. 
\begin{theo}\label{thtrivial}
If $B$ satisfies the Kreiss-Lopatinskii condition \eqref{UKL}, for $s\in [0,2]$, $(p_1,q_1)$ an admissible pair, 
\begin{equation*}
(u_0,f,g)\in H^s_0(\R^{d- 1}\times \R^+)\times \big(L^{p'_1}(\R_t^+,B^s_{q'_1,2})\cap 
B^s_{p'_1,2}(\R_t^+,L^{q'_1})\big)\times \mathcal{H}^s_0(\R^+),
\end{equation*}
(if $s=1/2$, $(u_0,g)\in H^{1/2}_{00}\times \mathcal{H}^{1/2}_{00}$), then 
the IBVP \eqref{IBVP} has a unique solution $u\in C(\R^+,H^s)$, 
moreover for any $(p,q)$ such that $p>2,\ \frac{2}{p}+\frac{d}{q}=\frac{d}{2}$,
it satisfies the a priori estimate 
\begin{equation*}
\|u\|_{L^p(\R_t^+,B^s_{q,2})\cap B^{s/2}_{p,2}(\R_t^+,L^q)}\lesssim \|u_0\|_{H^s}+
\|f\|_{L^{p'_1}(\R_t^+,B^s_{q'_1,2})\cap B^s_{p'_1,2}(\R_t^+,L^{q'_1})}+
\|g\|_{\mathcal{H}^s(\R_t^+)}. 
\end{equation*}
Moreover, solutions are causal, in the sense that if $(u_i)_{i=1,2}$ are solutions corresponding to initial data 
$(u_{0,i},f_i,g_i)$, such that $u_{0,1}=u_{0,2},\ f_1|_{[0,T]}=f_2|_{[0,T]}, g_1|_{[0,T]}=g_2|_{[0,T]}$, then 
$u_1|_{[0,T]}=u_2|_{[0,T]}$.
\end{theo}
\noindent 
Note that we have the usual range of indices for the integrability of $f$ but some time regularity is required. 
Such requirements are common for hyperbolic BVP (e.g. \cite{Met2} proposition 4.3.1), 
and the regularity required here is sharp in term of scaling, so that we are able to deduce the 
usual nonlinear well-posedness results from our linear estimates in section \ref{WPnonlin}. 
For the Dirichlet BVP, well-posedness with non trivial compatibility conditions 
holds:
\begin{theo}\label{maintheo}
In the case of Dirichlet boundary conditions, for $s\in [0,2]$, $(p_1,q_1)$ an admissible pair, 
$$(u_0,f,g)\in H^s(\R^{d- 1}\times \R^+)\times \big(L^{p'_1}(\R_t^+,B^s_{q'_1,2})\cap B^s_{p'_1,2}(\R_t^+,L^{q'_1})\big)
\times \mathcal{H}^s(\R^+_t),$$ 
that satisfy the compatiblity condition, then \eqref{IBVP} has a unique solution $u\in C(\R^+,H^s)$, 
moreover for any $(p,q)$ such that $p>2,\ \frac{2}{p}+\frac{d}{q}=\frac{d}{2}$
it satisfies the a priori estimate 
\begin{equation*}
\|u\|_{L^p(\R_t^+,B^s_{q,2})\cap B^{s/2}_{p,2}(\R_t^+,L^q)}\lesssim \|u_0\|_{H^s}+
\|f\|_{L^{p'_1}(\R_t^+,B^s_{q'_1,2})\cap B^s_{p'_1,2}(\R_t^+,L^{q'_1})}+
\|g\|_{\mathcal{H}^s(\R_t^+)}. 
\end{equation*}
\end{theo}

% In particular we construct in section \ref{WPnonlin} 
% the first example of global solutions with sharp boundary data, and discuss in section \ref{scattering} 
% the meaning of "asymptotically linear".

\paragraph{Plan of the article} In section \ref{funcanalysis} we recall a number of standard results on 
Sobolev spaces, and describe the $\mathcal{H}^s$ spaces:
completeness, duality, density properties etc. Section \ref{linestim} is devoted 
to the proof of theorems \ref{thtrivial} and \ref{maintheo}, it also contains the precise assumptions on the boundary 
operator $B$ and the definition of the Kreiss-Lopatinskii.
 In section \ref{WPnonlin} we prove the local well-posedness
in $H^1$ of nonlinear Dirichlet boundary value problems with classical restrictions on the nonlinearity, and 
global well-posedness for small data. Finally section \ref{scattering} is  devoted to the description of the long time 
 behaviour of the global small solutions: we prove that in some sense they do not behave differently from the 
restriction to $y\geq 0$ of  solutions of a linear Cauchy problem.
\section{Notations and functional background}\label{funcanalysis}
\subsection{Notations}
The Fourier transform of a function $u$ is denoted $\widehat{u}$. As we will use Fourier transform in the 
$(x,y)$ variable, $x$ variable or $(x,t)$ variable, we use when necessary the less ambiguous notation
$$\widehat{u}=\mathcal{F}_{x,t}u:=\int_{\R}\int_{\R^{d-1}} u(x,t)e^{-ix\cdot \xi -i\delta t}dx dt.$$ 
The notation $\R_t$ emphasizes the time variable. \\
Lebesgue spaces on a set $\Omega$ are denoted $L^p(\Omega)$. For $X$ a Banach space $L^p_tX:=L^p(\R_t,X)$ or 
depending on the context 
$L^p(\R_t^+,X)$, similarly $L^p_TX:=L^p([0,T],X)$. Similarly, $L^p_x$ refers to functions
defined on $\R^{d-1}$.
When dealing with nonlinear problems, we shall use the 
convenient but unusual notation $L^p=L^{1/p}$ (unambiguous as we work only with Lebesgue spaces with $p>1$).
\\
We write $a\lesssim b$ if $a\leq Cb$ with $C$ a positive constant. Similarly, $a\sim b$ if there 
exists $C_1,C_2>0$ such that $C_1a\leq b\leq C_2b$.
\subsection{Functional  spaces}\label{rappelsobo}
$\mathcal{S}'(\R^d)$ is the set of tempered distributions, dual of $\mathcal{S}(\R^d)$. $L^p(\Omega)$ is the 
Lebesgue space, we follow the usual notation $p':=p/(p-1)$.
For $s\in \R$,  
$$H^s(\R^d)=\bigg\{u\in \mathcal{S}'(\R^d):\int_{\R^d} (1+|\xi|^2)^s |\widehat{u}|^2d\xi<\infty\bigg\}.$$ 
$\dot{H}^s$ is the homogeneous 
Sobolev space. For $\Omega$ open, $H^s(\Omega)$ is defined as the set of restrictions to $\Omega$ 
of distributions in $H^s(\R^n)$, with the restriction norm
\begin{equation*}
\|u\|_{H^s(\Omega)}=\inf_{v\text{ extension of }u}\|v\|_{H^s(\R^d)}.     
\end{equation*}
Similarly, for $X$ a Banach space, $H^s(\Omega,X)$ denotes the Sobolev space of $X$ valued distributions.
We recall a few  facts (see e.g. \cite{lionsmagenes},\cite{lionsmagenes2}):
\begin{enumerate}
 \item For $n$ integer, $\Omega$ smooth simply connected, $H^n(\Omega,X)$ coincides topologically with 
 $\{u:\int_{\Omega}\sum_{|\alpha|\leq n}|\partial^\alpha u|^2dx\}$, that is 
$\|u\|_{H^n(\Omega)}\sim (\int_{\Omega}\sum_{|\alpha|\leq n}|\partial^\alpha v|^2dx)^{1/2}$, with constants
that depend on $\Omega, s$. If $\Omega=I$ is an interval the constants only depend on $1/|I|$ and $s$, in particular 
if $I$ is unbounded they only depend on $s$. The same is true if $\Omega$ is a half space.
\item For any $s\geq 0$, there exists a continuous 
extension operator $T_s:H^t(\Omega,X)\to H^t(\R^d,X)$ for $t\leq s$, moreover $T_s$ can be chosen such that it 
is valued into functions supported in $\{x:\ d(x,\Omega)\leq 1\}$.
% (for example, if $s=1$ $T_s$ is 
% simply obtained by symmetrization and truncation).
% For any integer $k\leq s$, it satisfies 
% $\|(T_su)^{(k)}\|_{L^2(\R,X)}\lesssim \|u^{(k)}\|_{L^2(I,X)}$ with constant only depending on $1/|I|$ and $s$.
If $s<1/2$, the zero extension is such an operator and in this case the operator's norm does not depend on 
$\Omega$.
\item $H^s_0(\Omega)$ is the closure in $\Omega$ of $C_c^\infty$. The extension by zero outside $\Omega$ is continuous 
$H^s_0(\Omega)\to H^s(\R^d)$ if $s\nequiv 1/2[\Z]$, but not if $s\equiv 1/2[\Z]$. The subset of $H^{1/2}$ on which the 
extension by zero is continuous is the so-called Lions-Magenes space $H^{1/2}_{00}$, see \cite{Tartar} section 33.
\end{enumerate}
For $n\in \N$, $W^{n,p}(\R^d)$ is the Sobolev space with norm $\sum_{|\alpha|\leq n}\int 
|\partial^\alpha u|^pdx$.
The Besov spaces on $\R^d$ are denoted $B^s_{p,q}(\R^d)$, they are defined by real interpolation \cite{berglof}
\begin{equation*}
\forall\,0\leq s\leq 2,\ B^s_{p,q}=[L^p,W^{2,p}]_{s/2,q}.
\end{equation*}
As for Sobolev spaces $B^s_{p,q}(\Omega)$ is defined by restriction. Due to the existence of extension operators, 
it is equivalent to define $B^s_{p,q}(\Omega)=[L^p(\Omega),W^{2,p}(\Omega)]_{s/2,q}$, the norm equivalence 
depends on $\Omega$. For $n\in \N$, the following inclusions stand (\cite{berglof} Theorem 6.4.4)
\begin{equation*}
\forall\,p\geq 2,\ B^n_{p,2}(\Omega)\subset W^{n,p}(\Omega),\ W^{n,p'}(\Omega)\subset B^n_{p',2}(\Omega).
\end{equation*}
The extension by zero outside $\Omega$ is often denoted $P_0$ (independently of $\Omega$), the restriction operator 
is denoted $R$.
\subsection{The \texorpdfstring{$\mathcal{H}^s$}{H} spaces}\label{secHs}
\paragraph{Structure and traces}
\begin{prop}\label{structureH}
 For $s\geq 0$, we define the space $\mathcal{H}^s(\R^{d-1}\times\R_t)$ as the set of tempered distributions 
 $g$ such that $\widehat{g}\in L^1_{\text{loc}}$ and 
 \begin{equation*}
\|g\|_{\mathcal{H}^s(\R^{d-1}_x\times \R_t)}^2:=\iint_{\R^{d-1}\times \R}
(1+|\xi|^2+|\delta|)^s\sqrt{||\xi|^2+\delta|} |\widehat{g}|^2d\delta d\xi<\infty.  
 \end{equation*}
When $d$ is unambiguous, we write for conciseness $\mathcal{H}^s(\R_t)$.\\
It is a complete Hilbert space, in which $C^\infty_c(\R^{d-1}_x\times \R_t)$ is dense, and has equivalent norm 
\begin{eqnarray*}
\|g\|_{\mathcal{H}^s}&:=&\bigg(\iint_{\R^{d-1}\times \R}
(1+|\xi|^2+|\delta|)^s\sqrt{||\xi|^2+\delta|} |\widehat{g}|^2d\delta d\xi\bigg)^{1/2}\\
&\sim& \bigg(\iint_{\R^{d-1}\times \R}
(1+|\xi|^{2s}+||\xi|^2+\delta|^s)\sqrt{||\xi|^2+\delta|} |\widehat{g}|^2d\delta d\xi\bigg)^{1/2}.
\end{eqnarray*}
The space $\mathcal{H}^0$ is denoted $\mathcal{H}$.
The map $u\mapsto \nabla_xu$ is continuous $\mathcal{H}^s\rightarrow \mathcal{H}^{s-1}$ for $s\geq 1$, and 
$u\mapsto \partial_tu$ is continuous $\mathcal{H}^s\rightarrow \mathcal{H}^{s-2}$ for $s\geq 2$.\\
For $s>1/2$, $\mathcal{H}^s\hookrightarrow C\big(\R_t,H^{s-1/2}(\R^{d-1}_x)\big)$, in particular for 
any $t\in \R$, the trace operator $g\mapsto g(\cdot,t)$ is continuous $\mathcal{H}^s\to H^{s-1/2}$.
\end{prop}
\begin{proof}
Obviously, $\mathcal{H}^s\subset \mathcal{H}^{s'}$ for $s>s'$. 
Let $g\in \mathcal{H}$, from Cauchy-Schwarz's inequality
\begin{eqnarray*}
\iint_{\R^{d-1}\times \R}|\widehat{g}(\xi,\delta)|(1+|\xi|+|\delta|)^{-d}d\xi d\delta &\leq& \|g\|_{\mathcal{H}}
\bigg(\iint  \frac{1}{(1+|\xi|+|\delta|)^{2d}\sqrt{||\xi|^2+\delta|}}d\xi d\delta\bigg)^{1/2}\\
&\lesssim&  \|g\|_{\mathcal{H}}
 \bigg(\int_{\R^{d-1}}  \frac{1}{(1+|\xi|)^{d+1}}d\xi d\delta\bigg)^{1/2}\lesssim \|g\|_{\mathcal{H}},
\end{eqnarray*}
thus the embedding $\mathcal{H}\hookrightarrow \mathcal{S}'$ is continuous. 
We define the measure $\mu$ by $d\mu=(1+|\xi|^2+|\delta|)^s\sqrt{||\xi|^2+\delta|}d\delta d\xi$. 
If $g_n$ is a Cauchy sequence in $\mathcal{H}^s$, $\widehat{g_n}$ is a Cauchy sequence in $L^2(d\mu)$. 
By completeness of Lebesgue spaces, there exists $v\in L^2(d\mu)$ such that 
$\|\widehat{g_n}-v\|\longrightarrow 0$. From the previous computations, $\mathcal{F}^{-1}(v)\in \mathcal{S}'$ 
and $\lim_{\mathcal{S}'}g_n=\mathcal{F}^{-1}v\in \mathcal{H}^s$.\\
The density of $C_c^\infty$ in $\mathcal{H}^s$ is obtained via the usual procedure. The equivalence of norms 
is a consequence of the elementary inequality $|a+b|^s\geq (1-2^{-1/s})^s(|a|^s-2|b|^s)$.
\vspace{1mm}\\
Let us now consider the trace problem. We start with the existence of a trace at $t=0$: 
\begin{eqnarray*}
g(x,0)&=&\int_{\R^{d-1}\times \R}e^{ix\cdot\xi}\widehat{g}(\xi,\delta) d\delta d\xi,\\
\Rightarrow 
\|g(\cdot,0)\|_{H^{s-1/2}}^2&=&\int_{\R^{d-1}}|(1+\xi|)^{2s-1}\bigg|\int_\R \widehat{g}d\delta\bigg|^2d\xi\\
&\leq& 
\int_{\R^{d-1}}\bigg( \int_\R|\widehat{g}|^2\sqrt{||\xi|^2+\delta|}(1+|\xi|^2+|\delta|)^sd\delta \\
&&\hspace{2cm}\times(1+|\xi|)^{2s-1}\int_\R\frac{1}{\sqrt{||\xi|^2+\delta|}(1+|\xi|^2+|\delta|)^s}d\delta\bigg) d\xi.
\end{eqnarray*}
Now clearly $\int_\R\frac{1}{\sqrt{||\xi|^2+\delta|}(1+|\xi|^2+|\delta|)^s}d\delta$ is bounded for
$|\xi|\leq 1$, and for $|\xi|\geq 1$ setting $\tau=|\xi|^2\mu$
\begin{equation*}
|\xi|^{2s-1}\int_\R\frac{1}{\sqrt{||\xi|^2+\delta|}(|\xi|^2+|\delta |)^s}d\tau\leq \int_\R 
\frac{1}{\sqrt{|1+\mu|}(1+|\mu|)^s}d\mu<\infty.
\end{equation*}
Therefore the trace at $t=0$ maps continuously $\mathcal{H}^s(\R_t)$ to $H^{s-1/2}(\R^{d-1})$. 
It is easily checked that the map $T_t:\ g\rightarrow g(\cdot,\cdot+t)$ is unitary $\mathcal{H}^s\rightarrow 
\mathcal{H}^s$ and for any $g\in \mathcal{H}^s,\ \lim_0 \|T_tg-g\|_{\mathcal{H}^s}=0$. 
Combining this observation with the existence of the trace at $t=0$ 
implies the embedding $\mathcal{H}^s\hookrightarrow C_tH^{s-1/2}$.
\end{proof}
Finally, we identify $(\mathcal{H}^s)'$ in a standard way:
\begin{prop}[Duality of $\mathcal{H}^s$ spaces]
For $s>0$, the topological dual $(\mathcal{H}^s)'$ is the set of tempered distributions $g'$ such that 
$\widehat{g'}\in L^1_{\text{loc}}$ and 
\begin{equation*}
\|g'\|_{(\mathcal{H}^s)'}^2=\iint_{\R^{d-1}\times \R_t}\frac{(1+|\xi|^2+|\delta|)^{-s}}{\sqrt{||\xi|^2+\delta|}}
|\widehat{g'}|^2d\delta d\xi<\infty,
\end{equation*}
$\mathcal{S}(\R^n)$ is dense in $(\mathcal{H}^s)'$, and $(\mathcal{H}^s)'$ acts on $\mathcal{H}^s$
with the $L^2$ duality bracket
\begin{equation*}
\langle g,\,g'\rangle_{\mathcal{H}^s,(\mathcal{H}^s)'}=\iint \widehat{g}\overline{\widehat{g'}}d\delta d\xi.
\end{equation*}
\end{prop}

\paragraph{Restrictions, extensions}

\begin{define}
For $s\geq 0$, $I$ an interval the space $\mathcal{H}^s(I)$ is the set of restrictions to 
$\R^{d-1}\times I$ of distributions in $\mathcal{H}^s(\R_t)$, with norm 
$\displaystyle \|g\|_{\mathcal{H}^s(I)}:=\inf_{\widetilde{g}\text{ extension}}
\|\widetilde{g}\|_{\mathcal{H}^s}$.\\
For $s\not \equiv 1/2[\Z]$, we define $\mathcal{H}^s_0=\mathcal{H}^s$ if $s<1/2$, else 
$$\mathcal{H}^s_0((a,b))=\{g\in \mathcal{H}^s((a,b)):\ 
\forall\, 0\leq 2k\leq [s-1/2],\ \lim_{a,b} \|\partial_t^k g(\cdot,t)\|_{H^{s-2k-1/2}}
=0\}.$$
\end{define} 

\noindent
Obviously, if $a$ (or $b$) is finite, the definition above simply amounts to 
$\partial_t^kg(\cdot,a)=0$.\\
A very convenient observation is that $\mathcal{H}^s$ is a kind of Bourgain space: let $\Delta'$ be the laplacian on
$\R^{d-1}$, we have using the change of variable $\delta -\xi^2=\mu$ 
\begin{eqnarray*}
\|e^{-it\Delta'}g\|_{\dot{H}_t^{(1+2s)/4}L^2_{x}\cap \dot{H}^{1/4}_tH^s}^2&=&\iint |\delta|^{1/2}
\big(1+|\delta|^s+|\xi|^{2s})\big|\mathcal{F}_{x,t}e^{-it\Delta}g\big|^2d\delta d\xi\\
&=&\iint |\delta|^{1/2}\big(1+|\delta|^s+|\xi|^{2s})|\widehat{g}(\xi,\delta-\xi^2)|^2d\delta d\xi\\
&\sim &\iint |\xi^2+\mu|^{1/2}\big(1+|\mu|^s+|\xi|^{2s})|\widehat{g}(\xi,\mu)|^2d\mu d\xi.
\end{eqnarray*}
so that $\|g\|_{\mathcal{H}^s}\sim \|e^{-it\Delta'}g\|_{\dot{H}^{(1+2s)/4}L^2_{x}\cap \dot{H}^{1/4}H^s}$. The following
results are elementary consequences of this remark and the classical theory on Sobolev spaces.

\begin{coro}\label{normeloc}
Let $I$  an interval, $g\in \mathcal{H}^s(I)$. We define the zero extension $P_0:g\mapsto P_0g$
\begin{equation*}
P_0g(\cdot,t)=
\left\{
\begin{array}{ll}
 g(\cdot,t)\text{ if }t\in I,\\
 0 \text{ else}.
\end{array}
\right.
\end{equation*}
We have the following assertions:
\begin{enumerate}
\item 
With constants only depending on $s$ 
\begin{equation*}
\|g\|_{\mathcal{H}^s(I)}\sim \|e^{-it\Delta'}g\|_{\dot{H}^{(2s+1)/4}(I,L^2)\cap \dot{H}^{1/4}(I,H^s)}.
\end{equation*}
 \item For any $s\geq 0$, there exists an extension operator $\mathcal{T}_s$ such that for $k\leq s$, 
 $\mathcal{T}_s:\mathcal{H}^k(I)\to \mathcal{H}^k(\R)$ is continuous and 
 for any $g\in \mathcal{H}^s(I),$ $\mathcal{T}_sg(t)=0$ for $t\notin (\inf I-1,\ \sup I+1)$.\\
 If $s<1/2$, $P_0$ is such an operator.
\item For $s\geq 0$, $g\in \mathcal{H}^s(\R)$, then 
$\displaystyle \lim_{T\rightarrow \infty}\|g\|_{\mathcal{H}^s([T,\infty[)}=0$.
\item For $s\geq 0$, $\mathcal{H}^s_0(\R)=\mathcal{H}^s$, moreover if 
$s\not\equiv 1/2[\Z]$ $P_0$ is continuous $\mathcal{H}^s_0(I)\rightarrow \mathcal{H}^s(\R)$.
\item The restriction operator $(\mathcal{H}(\R))'\rightarrow (\mathcal{H}(I))',\ g\mapsto P_0^*(g)$ is a continuous 
surjection.
\end{enumerate}
\end{coro}

\begin{proof}
$1.$ is a direct consequence of the definition of Sobolev spaces by restriction.\\
$2.$ According to paragraph 
\ref{rappelsobo}, there exists an extension operator $T$ such that 
\begin{equation*}
\|T(e^{-it\Delta'}g)\|_{\dot{H}^{(1+2s)/4}(\R,L^2_{x})\cap \dot{H}^{1/4}(\R,H^s)}\lesssim 
\|e^{-it\Delta'}g\|_{\dot{H}^{(1+2s)/4}(I,L^2_{x})\cap \dot{H}^{1/4}(I,H^s)}\lesssim \|g\|_{\mathcal{H}^s(I)}.
\end{equation*}
It is then clear that $\mathcal{T}=e^{it\Delta'}T(e^{-it\Delta'})$ defines a continuous extension 
operator. \\
$3.$ If $r$ is an integer, $\lim_{T\rightarrow \infty}\|f\|_{H^r([T,\infty[)}= 0$ (note that this is true also 
for the restriction norm), then we can conclude by 
a density argument and the inequality 
\begin{equation*}
\|e^{-it\Delta}g\|_{\dot{H}^{(1+2s)/4}(I,L^2_{x})\cap \dot{H}^{1/4}(I,H^s)}\leq 
\|e^{-it\Delta}g\|_{H^k(I,L^2)\cap H^1(I,H^s)},\ k\geq (1+2s)/4. 
\end{equation*}
$4.$ Let $g\in \mathcal{H}^s(\R)$. By continuity of the trace and point 3 
$$\lim_{\infty}\|\partial_t^kg(\cdot,t)\|_{H^{s-2k-1/2}}\lesssim
\lim_{\infty}\|g\|_{\mathcal{H}^s([T,\infty))},$$
the limit at $-\infty$ follows from a symmetry argument. \\
Now fix $a\in \R$. If for $0\leq 2k\leq s-1/2$, $\partial_t^kg(\cdot,a)=0$, this 
implies clearly $\partial_t^k(e^{-it\Delta}g)(\cdot,a)=0$, so that we can apply the continuity of the extension 
by $0$ for $e^{-it\Delta'}g$ in the usual Sobolev spaces.\\
5. Continuity follows from point 4, the surjectivity from the definition of $\mathcal{H}(I)$.
\end{proof}

Similarly to the  Sobolev space $H^{1/2}(\R^+)$, the zero extension is \emph{not} continuous 
$\mathcal{H}^{1/2}(\R^+)\to \mathcal{H}^{1/2}(\R)$. Nevertheless, we observe that 
$P_0g\in \mathcal{H}^{1/2}(\R)$ if $e^{-it\Delta'}P_0g=P_0e^{-it\Delta'}g\in \dot{H}^{1/2}L^2\cap \dot{H}^{1/4}H^{1/2}$, 
which is true if $e^{-it\Delta'}g\in \dot{H}^{1/2}(\R^+,L^2)\cap \dot{H}^{1/4}(\R^+,H^{1/2})$ and 
(see \cite{Tartar} section 33)
\begin{equation}\label{condition00}
I(g):=\int_{\R^+\times \R^{d-1}}\frac{|e^{-it\Delta'}g(x,t)|^2}{t}dtdx<\infty.
\end{equation}
Or more compactly $e^{-it\Delta'}g\in \dot{H}^{1/2}_{00}(\R^+,L^2)\cap \dot{H}^{1/4}(\R^+,H^{1/2})$, 
endowed with the norm 
\begin{equation*}
 \|e^{-it\Delta'}g\|_{\dot{H}^{1/2}_{00}L^2\cap \dot{H}^{1/4}H^{1/2}}:=
 \|e^{-it\Delta'}g\|_{\dot{H}^{1/2}L^2\cap \dot{H}^{1/4}H^{1/2}}+I(g)^{1/2}.
\end{equation*}
% (for a quick reminder on $H^{1/2}_{00}$ we refer to \cite{Tartar} section 33). 
These observations lead to 
the following definition:

\begin{define}\label{def00}
We denote $\mathcal{H}^{1/2}_{00}(\R^+):=\{g\in \mathcal{H}^{1/2}(\R^+):\ P_0g\in \mathcal{H}^{1/2}(\R)\}$, 
it  coincides with 
$\{g:\ e^{-it\Delta}g\in \dot{H}^{1/2}_{00}\cap \dot{H}^{1/4}H^{1/2}\}$, and is a Banach space for the norm 
\begin{equation}\label{norme00}
 \|g\|_{\mathcal{H}^{1/2}_{00}}=
 \|e^{-it\Delta'}g\|_{\dot{H}^{1/2}L^2\cap \dot{H}^{1/4}H^{1/2}}+I(g)^{1/2}.
\end{equation}
\end{define}

\begin{rmq}
Of course we could also define $\mathcal{H}^{1/2}_{00}(I)$, but it is not useful for this paper.
\end{rmq}

\paragraph{Interpolation}
For basic definitions of interpolation, we refer to \cite{berglof}, sections 3.1 and 4.1.
We denote $[\cdot,\cdot]_{\theta}$ the complex interpolation functor and 
$[\cdot,\cdot]_{\theta,2}$ the real interpolation functor with parameter $2$.

\begin{prop}
For $s_0,s_1\geq 0$, $0<\theta<1$ we have
\begin{eqnarray*}
[\mathcal{H}^{s_0},\mathcal{H}^{s_1}]_{\theta}&=&\mathcal{H}^{(1-\theta)s_0+ \theta s_1}\text{ (complex interpolation) },
\\
\ [\mathcal{H}^{s_0},\mathcal{H}^{s_1}]_{\theta,2}&=&
\mathcal{H}^{(1-\theta)s_0+\theta s_1}\text{ (real interpolation) }.
\end{eqnarray*}
\end{prop}

\begin{proof}
By Fourier transform we are reduced to the interpolation of weighted $L^2$ spaces. 
For real interpolation, this is theorem $5.4.1$ of \cite{berglof}, for complex interpolation this is theorem 
$5.5.3$.
\end{proof}
The interpolation of $\mathcal{H}^s_0$ spaces is a bit more delicate.
\begin{prop}\label{interpHs0}
For $0<\theta<1$, $\theta \neq 1/4$, $I$ an interval we have
\begin{eqnarray*}
[\mathcal{H}_0(I),\mathcal{H}^{2}_0(I)]_{\theta}&=&\mathcal{H}^{2\theta }_0(I)\text{ (complex interpolation) },
\\
\ [\mathcal{H}_0(I),\mathcal{H}^{2}_0(I)]_{\theta,2}&=&
\mathcal{H}^{2\theta }_0(I)\text{ (real interpolation) }.
\end{eqnarray*}
If $s_0=0,\ s_1=2, \theta=1/4$, then 
\begin{eqnarray*}
[\mathcal{H}_0(\R^+),\mathcal{H}^{2}_0(\R^+)]_{1/4}&=&\mathcal{H}^{1/2}_{00}(\R^+)\text{ (complex interpolation) },
\\
\ [\mathcal{H}(\R^+),\mathcal{H}^{2}_0(\R^+)]_{1/4,2}&=&
\mathcal{H}^{1/2}_{00}(\R^+)\text{ (real interpolation) }.
\end{eqnarray*}
\end{prop}
\begin{proof}

% Consider the map 
% \begin{equation*}
% g\in \mathcal{H}^s_0(\R^+)\mapsto u=
% \left\{
% \begin{array}{ll}
%  \text{the solution of }\eqref{BVP} \text{ for }t\geq 0,\\
%  0 \text{ for }t<0.
% \end{array}
% \right.
% \end{equation*}
% We are left to prove 
% $[\mathcal{H},\mathcal{H}^2_0]_{s,2}=\mathcal{H}^s_0$.
We only detail the case $I=\R^+$, the case of a general interval is similar.
According to corollary \ref{normeloc}, for $s\in[0,2]\setminus\{1/2\}$ the zero extension $P_0$, 
resp. the restriction $\mathcal{R}$ to $\R^+$, is a 
continuous operators $\mathcal{H}^s_0(\R^+)\rightarrow \mathcal{H}^s(\R)$, 
resp. $\mathcal{H}^s(\R)\to \mathcal{H}^s(\R^+)$, with $\mathcal{R}\circ P_0=Id$. Therefore by interpolation
\begin{equation*}
P_0\big([\mathcal{H}(\R^+),\mathcal{H}^2_0(\R^+)]_{s,2}\big)\subset \mathcal{H}^{2s}(\R),
\end{equation*}
and from the existence of traces, if $s>1/4$, for $g\in [\mathcal{H},\mathcal{H}^2_0]_{s,2}$, 
$g(0)=\lim_{0^-}P_0g(t)=0$, thus $[\mathcal{H},\mathcal{H}^2_0]_{s,2}\subset \mathcal{H}^{2s}_0(\R^+)$.
Conversely, for $g\in \mathcal{H}^s(\R)$, we define 
$$Sg:\ t\in (0,\infty)\rightarrow g(t)-3g(-t)+2g(-2t).$$ Clearly, it is continuous 
$\mathcal{H}^s(\R)\to \mathcal{H}^s(\R^+)$, and when it makes sense $Sg(0)=0$, $\partial_tSg(0)=0$
thus it is $\mathcal{H}^s_0(\R^+)$ valued.
By interpolation $S$ is continuous $\mathcal{H}^{2s}(\R)\rightarrow 
[\mathcal{H}(\R^+),\mathcal{H}^2_0(\R^+)]_{s,2}$. Now we can observe that $S\circ P_0=Id$ on 
$\mathcal{H}^s_0(\R^+)$, therefore $\mathcal{H}^{2s}_0(\R^+)\subset [\mathcal{H},\mathcal{H}^2_0]_{s,2}$ and 
the identification is complete. 
\\
If $s=1/2$, we observe that the same argument can be applied provided $P_0$ acts continuously 
$\mathcal{H}^{1/2}_{00}(\R^+)\to \mathcal{H}^{1/2}(\R)$, but this is true according to definition \ref{def00}. 
 
\end{proof}

\subsection{Interpolation spaces and composition estimates} 
In order to treat nonlinear problems, estimates in $B^s_{p,2}L^q$ require some composition estimates.

\begin{prop}
Let $A$ be a Banach space. For $0<\theta<1$,
$[L^p(\R,A),W^{1,p}(\R,A)]_{\theta,2}=B^\theta_{p,2}(\R,A)$ the fractional Besov space endowed with the norm
\begin{equation*}
\|u\|_{B^\theta_{p,2}A}^2:=\int_0^\infty \bigg(\frac{\|u(\cdot+h)-u(\cdot)\|_{A}}{h^\theta}\bigg)^2\frac{dh}{h}
+\|u\|_{L^pA}^2:=\|u\|_{\dot{B}^\theta_{p,2}}^2+\|u\|_{L^pA}^2.
\end{equation*}
\end{prop}

For completeness we include a short proof in the spirit of \cite{Tartar} of this well-known result.

\begin{proof}
We use the K-method for interpolation. Let $K(h)=\inf_{u=u_0+u_1}\|u_0\|_{L^pA}+h\|u_1\|_{W^{1,p}A}$. 
If $u\in [L^p(\R,A),W^{1,p}(\R,A)]_{\theta,2}$, then for any $h\geq 0$ there exists 
$(u_0,u_1)$ with $u=u_0+u_1,\ \|u_0\|_{L^pA}+h\|u_1\|_{W^{1,p}A}\leq 2K(h)$ and 
$\|u\|_{[L^pA,W^{1,p}A]_{\theta,2}}:=(\int_0^\infty (K(h)/h^\theta)^2dh/h)^{1/2}<\infty$. The standard estimate 
$\|u_1(\cdot+h)-u_1(\cdot)\|_{L^p}\leq h\|u_1\|_{W^{1,p}}$ implies 
$$
\int_0^\infty \bigg(\frac{\|u(\cdot+h)-u(\cdot)\|_{L^pA}}{h^\theta}\bigg)^2\frac{dh}{h}\leq 4\int_0^\infty 
\bigg(\frac{K(h)}{h^\theta}\bigg)^2\frac{dh}{h}.
$$
Conversely, assume the left hand side of the equation above is finite and $u\in L^pA$. 
For $h>0$, $\rho_h=\rho(\cdot/h)/h$ with $\rho\in C_c^\infty,\ \rho\geq 0,\ \int \rho=1$, 
$\text{supp}(\rho)\subset [-1,1]$,
we set $u_0=u-\rho_h*u, u_1=\rho_h*u$. Minkowski's inequality gives
\begin{equation*}
\|u-\rho_h*u\|_{L^pA}\leq \int_{-h}^h\rho_h(s)\|u(\cdot)-u(\cdot-s)\|_{L^pA}ds\lesssim \frac{1}{h}\int_0^h
\|u(\cdot+s)-u(\cdot)\|_{L^pA}ds,
\end{equation*}
\begin{equation*}
\|(\rho_h)'*u\|_{L^p A}\leq \int_{-h}^h\|\rho_h'(s)\big(u(\cdot-s)-u(\cdot)\big)\|_{L^pA}ds
\lesssim \frac{1}{h^2}\int_0^h\|u(\cdot+s)-u(\cdot)\|_{L^pA}ds,
\end{equation*}
therefore $K(h)\leq \|u-\rho_h*u\|_{L^pA}+h\|\rho_h*u\|_{W^{1,p}A}\lesssim h\|u\|_{L^pA}+\frac{1}{h}
\int_0^h\|u(\cdot+s)-u(\cdot)\|_{L^pA}ds$. Also, it is obvious that for $h\geq 1$, $K(h)\leq \|u\|_{L^p}$. 
By integration
\begin{equation*}
\int_0^\infty  \bigg(\frac{K(h)}{h^\theta}\bigg)^2\frac{dh}{h}\lesssim \|u\|_{L^pA}^2+\int_0^1
\bigg(\int_0^h\|u(\cdot+s)-u(\cdot)\|_{L^pA}ds\bigg)^2\frac{dh}{h^{3+2\theta}}.
\end{equation*}
We set $f(h)=\|u(\cdot+h)-u(\cdot)\|_{L^pA},\ F(h)=\int_0^hfds$. An 
integration by parts and Cauchy-Schwarz's inequality gives 
\begin{eqnarray*}
\int_0^\infty\big(F(h)\big))^2\frac{dh}{h^{3+2\theta}}
&= &\frac{2}{2+2\theta}\int_0^\infty f(h)F(h)\frac{dh}{h^{2+2\theta}}\\
&\leq& \frac{2}{2+\theta}\bigg(\int_0^\infty \bigg(\frac{f(h)}{h^\theta}\bigg)^2\frac{dh}{h}\bigg)^{1/2}
\bigg(\int_0^\infty F(h)^2\frac{dh}{h^{3+2\theta}}\bigg)^{1/2},
\end{eqnarray*}
from which we deduce $\displaystyle \int_0^\infty\bigg(\frac{K(h)}{h^\theta}\bigg)^2\frac{dh}{h}
\lesssim \|u\|_{L^pA}^2+\int_0^\infty\bigg(\frac{\|u(\cdot+h)-u(\cdot)\|_{L^pA}}{h^\theta}\bigg)^2\frac{dh}{h}$.
\end{proof}

\begin{prop}\label{compoFrac}
Let $F:\ \C\to \C$ such that $|F(u)|\lesssim |u|^a,\ |F'(u)|\lesssim |u|^{a-1}$, $a>1$. Then 
\begin{equation*}
\|F(u)\|_{B^s_{p,2}(\R_t,L^q)}\lesssim \|u\|_{L^{p_1}(\R_t,L^{q_1})}^{a-1}\|u\|_{B^s_{p_2,2}(\R_t,L^{q_2})},
\end{equation*}
with 
\begin{equation*}
p_1,q_1,p_2,q_2\geq 1,\ \frac{1}{q}=\frac{a-1}{q_1}+\frac{1}{q_2},\ \frac{1}{p}=\frac{a-1}{p_1}+\frac{1}{p_2}.
\end{equation*}
\end{prop}

\begin{proof}
The $L^p_tL^q$ part of the norm is simply estimated with H\"older's inequality on $|u|^{a-1}\times |u|$.
For the $\dot{B}^s_{p,2}$ part, let $1/p=1/p_3+1/p_2,\ 1/q=1/q_3+1/q_2$:
\begin{eqnarray*}
\int_0^\infty &\bigg(&\frac{\|F(u)(\cdot+h)-F(u)(\cdot)\|_{L^p_tL^q}}{h^s}\bigg)^2\displaystyle\frac{dh}{h}\\
&\lesssim &\int_0^\infty \bigg(\frac{\|(|u(\cdot+h)|^{a-1}+|u(\cdot)|^{a-1})|u(\cdot+h)-u(\cdot)|
\|_{L^p_tL^q}}{h^s}\bigg)^2\frac{dh}{h}\\
&\lesssim &\int_0^\infty \bigg(\frac{\|(u^{a-1}\|_{L^{p_3}_tL^{q_3}}\|u(\cdot+h)-u(\cdot)\|_{L^{p_2}_tL^{q_2}}}{h^s}\bigg)^2
\frac{dh}{h}\\
&=&\|u\|_{L^{p_1}_tL^{q_1}}^{2(a-1)}\int_0^\infty \bigg(\frac{\|u(\cdot+h)-u(\cdot)\|_{L^{p_2}_tL^{q_2}}}{h^s}\bigg)^2
\frac{dh}{h}\\
&\leq&\|u\|_{L^{p_1}_tL^{q_1}}^{a-1}\|u\|_{B^s_{p_2,2}L^{q_2}}^2.
\end{eqnarray*}
\end{proof}
% Finally, let us recall a classical extension property for Besov spaces, see e.g. \cite{Triebel2}, chapter 2 section 10
% (for a self contained proof in the case $H^s=B^s_{2,2}$ instead of $B^s_{p,2}$, see \cite{lionsmagenes}, chapter 1 
% theorem 11.4, the argument works also for $B^s_{p,p}$). 
% \begin{prop}\label{zerobesov}
% If $1\leq p<2$, for any interval $I$ the zero extension outside of $I$ maps continuously 
% $B^{1/2}_{p,2}(I,L^q)\to B^{1/2}_{p,2}(\R,L^q)$ with continuity constant independent of $I$.
% \end{prop}

% \textcolor{red}{bug}
% There is a last ambiguity that should be lifted: we define local Besov spaces $B^s_{p,2}(I)$ by restriction, 
% and it is easy to check (see \cite{Tartar} section 36) that the restriction norm is equivalent to the 
% local Besov norm 
% \begin{equation*}
% \vvvert u\vvvert_{B^s_{p,2}(I)}=\bigg(\iint_{I^2}\frac{|u(x)-u(y)|^p}{|x-y|^{N+sp}}dxdy\bigg)^{1/p}.
% \end{equation*}
% However the constant of equivalence might depend on the size of $I$. Nevertheless, since the zero extension defines 
% a continuous operator $B^s_{p,2}(I)\to B^s_{p,2}(\R)$ with continuity constant 
% independent of $|I|$ if $sp<1$ we have:
% \begin{coro}\label{compoint}
% With the same settings as proposition \ref{compoFrac}, for $I$ an interval, $P_0$ the zero extension, if $sp<1$, 
% we have with constant independent of $|I|$
% \begin{equation*}
% \|P_0 F(u)\|_{B^s_{p,2}(\R,L^q)}\lesssim \|u\|_{L^{p_1}(I,L^{q_1})}\vvvert u\vvvert_{B^s_{p_2,2}(I,L^{q_2})},
% \end{equation*}
% \end{coro}

Finally, as the nonlinear problems require to construct local solutions, we shall use the following extension lemma. 

\begin{lemma}\label{extOp}
Let $p\geq 1,\ 0<s<1$ with $sp>1$, $A$ a Banach space. For any $0<T\leq 1$, there exists an extension operator $P_T:\ 
B^s_{p,2}([0,T],A)\to B^s_{p,2}(\R_t,A)$ such that $P_Tu(\cdot,t)=0$ if $t\notin [-T,2T]$ and 
(with constants unbounded as $sp\to 1$)
\begin{equation}\label{actionPT}
\left\{
\begin{array}{ll}
 \|P_Tu\|_{L^{p}(\R_t,A)}\lesssim \|u\|_{L^{p}([0,T],A)},\\
 \|P_Tu\|_{B^s_{p,2}(\R_t,A)}\lesssim T^{1/p-s}\|u\|_{B^s_{p,2}([0,T],A)}.
\end{array}\right.
 \end{equation}
\end{lemma}

\begin{proof}
We fix $\chi\in C_c^\infty([0,1[)$, $\chi(0)=1$, and define the operator 
\begin{equation*}
P_1:\ B^s_{p,2}([0,1],A)\to B^s_{p,2}(\R,A)
u\mapsto 
\left\{
\begin{array}{ll}
u(t),\ 0\leq t\leq 1,\\
u(2-t)\chi(t-1),\ 0\leq t\leq 2,\\
u(-t)\chi(-t),\ -1\leq t\leq 0,\\
0,\ \text{else}.
\end{array}
\right.
\end{equation*}
It is not difficult to check that $P_1$ is bounded $L^p([0,1],A)\to L^p_tA$, 
$W^{1,p}([0,1],L^q)\to W^{1,p}_tL^q$, with bounds independent of $p$, thus it is also bounded 
$B^s_{p,2}([0,1],A)\to B^s_{p,2}(\R,A)$.
Let $D_\lambda$ be the dilation operator $D_\lambda:u\mapsto u(\cdot,\lambda \cdot)$, we set 
\begin{equation*}
P_T=D_{1/T} \circ P_1 \circ D_T.
\end{equation*}
From a direct computation, $\|P_Tu\|_{L^p_tA}\leq 3\|\chi\|_\infty \|u\|_{L^p_TA}$, thus we are left to prove the 
second inequality in \eqref{actionPT}.\\
As $sp>1$, by Sobolev's embedding 
% $B^s_{p,2}\hookrightarrow L^\infty$, we have 
$\|D_Tu\|_{L^\infty ([0,1],A)}\lesssim \|u\|_{B^s_{p,2}([0,T],A)}$ thus
\begin{equation*}
\|P_1D_Tu\|_{L^p_tL^q}\lesssim \|P_1D_Tu\|_{L^\infty A}\lesssim \|D_Tu\|_{L^\infty([0,1],A)}
\lesssim \|u\|_{B^s_{p,2}([0,T],L^q)} .
\end{equation*}
 On the other hand  for $v$ an extension of $u$, basic computations give
\begin{equation*}
\|D_Tv\|_{\dot{B}^s_{p,2}L^q}=\int_0^\infty \frac{\|(D_Tv)(t+h)-(D_Tv)(t)\|_{L^p_tL^q}^2}{h^{1+2s}}dh\leq  
T^{2s-2/p}\|v\|_{B^s_{p,2}}^2,
\end{equation*}
thus for $T\leq 1$, $\|P_1D_Tu\|_{B^s_{p,2}(\R_t,L^q)}\lesssim \|u\|_{B^s_{p,2}([0,1],L^q)}$, from which we get with the same scaling 
argument $\|D_{1/T}P_1D_Tu\|_{B^s_{p,2}(\R_t,L^q)}\lesssim T^{1/s-1/p}\|u\|_{B^s_{p,2}([0,1])}$.
\end{proof}

\section{Linear estimates}\label{linestim}
The plan to solve \eqref{IBVP} is based on a superposition principle: let us denote abusively $u_0$ an extension 
of $u_0$ to $\R^d$. If we can solve the Cauchy problem
\begin{equation*}
\left\{
\begin{array}{ll}
i\partial_tv+\Delta v=f,\\
v|_{t=0}=u_0,
\end{array}
\right. (x,y,t)\in \R^d\times \R.
\end{equation*}
and the boundary value problem
\begin{equation}\label{BVPaux2}
\left\{
\begin{array}{ll}
i\partial_tw+\Delta w=0,\\
w|_{t=0}=0,\\
B(w|_{y=0},\partial_yw|_{y=0})=g-B(v|_{y=0},\partial_yv|_{y=0}),
\end{array}
\right. (x,y,t)\in \R^{d-1}\times \R^+\times \R^+.
\end{equation}
then $v|_{y\geq 0}+w$ is the solution to \ref{IBVP}. For this strategy to be fruitful we need a number of results: Strichartz 
estimates for $v$, trace estimates for $v|_{y=0},\partial_yv|_{y=0}$, existence and Strichartz estimates for $w$.
This is the program that we follow through section $3$.

\subsection{The pure boundary value problem}
Consider the linear boundary value problem 
\begin{equation}\label{BVP}
 \left\{
 \begin{array}{ll}
  i\partial_tu+\Delta u=0,\\
  B(u|_{y=0},\partial_{y}u|_{y=0})=g,\\
  u(\cdot,0)=0.
 \end{array}
\right.
\end{equation}
We use the following notion of solution (slightly stronger than definition \ref{defsol}):
\begin{define}\label{defsol2}
Let $g\in \mathcal{H}^s_0(\R^+)$. 
We say that $u$ is a solution of the BVP \eqref{BVP} if $u\in C(\R^+,H^s)$, 
there exists a sequence $g_n\in \cap_{k\geq 0}H^k_0(\R^{d-1}\times \R_t^+)$ with 
$\|g-g_n\|_{\mathcal{H}^s}\to_n 0$ and smooth solutions $u_n\in C^\infty(\R^+, \cap_{k\geq 0}H^k)$ of \eqref{BVP} 
with boundary data $g_n$ such that $\|u-u_n\|_{L^\infty H^s}\to_n 0$.
\end{define}
\paragraph{The Kreiss-Lopatinskii condition}
We recall the notation of 
the introduction
\begin{equation*} 
\mathcal{L}(B(a,b))(\xi,\tau)=b_1\mathcal{L}a(\xi,\tau)+b_2\mathcal{L}b(\xi,\tau),
\end{equation*}
with $b_1,b_2$ anisotropically homogeneous: $b_1(\lambda\xi,\lambda^2\tau)=b_1(\xi,\tau),\ b_2(\lambda \xi,\lambda^2\tau)
=\lambda^{-1}b_2(\xi,\tau)$.
Of course, the operator $B$ must satisfy some conditions. 
First of all, it should be defined independently 
of $\text{Re}(\tau):=\gamma>0$, so according to Paley-Wiener's theorem we assume that $b_1,b_2$ are holomorphic in 
$\tau$ on $\{(\tau,\xi)\in \C\times \R^{d-1},\ \text{Re}(\tau)>0\}$. Moreover we assume that $b_1$ extends continuously 
on $\{(i\delta,\xi)\in (\R\times \R^{d-1})\setminus\{0\}\}$, and a.e. in $(\delta,\xi)$, 
$\displaystyle\lim_{\gamma\to 0}b_2(\xi,\gamma+i\delta)$ exists.\\
The Kreiss-Lopatinskii condition is an algebraic condition that we introduce with the following heuristic:
assume that the solution $u$ belongs to $C_b(\R_t^+,\mathcal{S}(\R^{d-1}\times \R^+))$, and consider its Fourier-Laplace 
transform $\mathcal{L}u(\xi,y,-i\tau)=\iint e^{-\tau t+i\xi x}u(x,y,t)dx\,dt$. Then $\mathcal{L}u$ 
satisfies
\begin{equation*}
\partial_y^2\mathcal{L}u=(|\xi|^2-i\tau )\mL u.
\end{equation*}
The condition $\lim_{y\to \infty} \mathcal{L}u(y)=0$ imposes
\begin{equation}\label{solEDO}
\mL u=e^{-\sqrt{|\xi|^2-i\tau}y}\mL u(y=0).
\end{equation}
Here,  $\sqrt{\cdot}$ is the square root defined on $\C\setminus i\R^+$ such that $\sqrt{-1}=-i$. 
From \eqref{solEDO}, the condition $B(u|_{y=0},\partial_yu|_{y=0})=g$ rewrites 
$(b_1-\sqrt{|\xi|^2-i\tau}b_2)\mL u(0)=\mL g$, so that $\mL u(0)$ is uniquely determined from $\mathcal{L}g$ 
if 
\begin{equation}\label{UKL}
\exists\,\alpha,\beta>0:\ \forall\,(\gamma,\delta,\xi)\in \R^{+}\times \R\times \R^{d-1},\
\alpha\leq \bigg|\begin{pmatrix}b_1\\b_2\end{pmatrix}
\cdot \begin{pmatrix}1 \\ -\sqrt{|\xi|^2-i\tau}\end{pmatrix}\bigg|\leq \beta.
\end{equation}
\begin{define}
$B$ satisfies the (generalized) Kreiss-Lopatinskii condition if \eqref{UKL} is true.
\vspace{2mm}
\end{define}
\noindent
By homogeneity $b_1$ is uniformly bounded, thus \eqref{UKL} implies that 
$b_2\sqrt{|\xi|^2-i\tau}$ is uniformly bounded for $\text{Re}(\tau)\geq 0$, although $b_2$ may be infinite at some 
points $(\xi,i\delta)$. The vector
$V_-:=\begin{pmatrix}1 \\ -\sqrt{|\xi|^2-i\tau}\end{pmatrix}$ is the so-called stable eigenvector, 
and algebraically \eqref{UKL} means that the symbol of $B$, as a linear operator $\C^2\to \C$, defines an isomorphism 
$\text{span}(V_-)\to \C$.\\ 
Obviously, the Dirichlet boundary condition $b_D:=(1,0)$ satisfies the uniform Kreiss Lopatinskii condition.
It is also possible to include the Neuman boundary condition as well as the transparent boundary 
condition into this framework by setting 
\begin{eqnarray}
\label{BNeuman}
\mathcal{L}B_{N}(a,b)&=&\frac{\mathcal{L}b(\xi,\tau)}{\sqrt{|\xi|^2-i\tau}}\text{ (Neuman)},\\
\label{Btransparent}
\mathcal{L}B_{T}(a,b)&=&\mathcal{L}a(\xi,\tau)-\frac{\mathcal{L}b(\xi,\tau)}{\sqrt{|\xi|^2-i\tau}}\text{ (Transparent)}.
\end{eqnarray}
With this convention, $b_N\cdot V_-=-1$ and $b_T\cdot V_-=2$, so that both satisfy the Kreiss-Lopatinskii condition. 
Let us point out that in the case of Neuman boundary conditions, $B_N(u,\partial_yu)\in \mathcal{H}$ is equivalent to 
$\partial_yu|_{y=0}\in \mathcal{H}'$, indeed 
\begin{equation}\label{traceNeuman}
\|P_0g\|_{\mathcal{H}(\R)}^2=
\int_{\R^d} \frac{|\mathcal{L}(\partial_yu|_{y=0})|^2}{||\xi|^2+\delta|}\sqrt{||\xi|^2+\delta|}d\xi d\delta
=\|P_0\partial_yu|_{y=0}\|^2_{\mathcal{H}'} .
\end{equation}
% Finally, we point out that this definition is not completely intrinsic: it relies on our choice that 
% $b_2$ is anisotropically homogeneous of \emph{negative order}. 
% \begin{equation*}
% \forall\,(\mu,\gamma,\delta,\xi)\in (\R^+)^2\times \R\times \R^{d-1},\ b_2(\lambda \xi, \lambda^2(\gamma+i\delta))
% =\lambda^{-1}b_2( \xi, \gamma + i\delta).
% \end{equation*}
% If we had chosen $(b_1,b_2)$ homogeneous with respective order $1$ and $0$ (which is the classical 
% choice for hyperbolic problems), then the Neuman boundary 
% conditions would have had for symbol $(0,1)$, and in this case the Kreiss-Lopatinskii condition fails at
% $\delta=-|\xi|^2$. 
% This is a natural assumption in view of the scaling of the operator $i\partial_t+\Delta$.
% If this condition is satisfied, from a compactness argument \eqref{UKL} implies there exists some $\alpha(\gamma)>0$ on 
% $\R^+$
% such that $|b\cdot V_-|\geq \alpha(\gamma)>0$ uniformly. Therefore for fixed $\gamma$ a control of $g$ gives a control of 
% $\mathcal{L}u(0)$. 
\paragraph{The Kreiss-Lopatinskii condition and the backward BVP} \label{WPbackwards}
For general boundary conditions, the boundary 
value problem is not alwats reversible. Indeed if we solve 
\eqref{BVP} for $t\leq 0$, $g$ supported in $\R_t^-$, the parameter $\gamma$ in the Laplace transform is negative 
therefore the appropriate square root in formula \eqref{solEDO} is defined 
on $\C\setminus i\R^-$, and maps $-1$ to $i$. Let us denote it $\text{sq}$. 
Even if we dismiss analyticity issues, there is no reason that ``backward \eqref{UKL}'' stands
\begin{equation}
\exists\,\alpha,\beta>0:\ \forall\,(\gamma,\delta,\xi)\in \R^{-}\times \R\times \R^{d-1},\
\alpha\leq \bigg|\begin{pmatrix}b_1\\ b_2\end{pmatrix}\cdot \begin{pmatrix}1 \\ -\text{sq}(|\xi|^2-i\tau)\end{pmatrix}
\bigg|\leq \beta.
\end{equation}
For example, take the forward transparent  boundary condition $(b_1,b_2)=(1,\frac{-1}{\sqrt{|\xi|^2+\delta}})$, then 
\begin{equation*}
\forall\,(\xi,\delta)\text{ such that }|\xi|^2+\delta<0,\ 
\begin{pmatrix}
1\\
\dfrac{-i}{\sqrt{||\xi|^2+\delta|}}
\end{pmatrix}\cdot 
\begin{pmatrix}
 1\\
 -i\sqrt{||\xi|^2+\delta|}
\end{pmatrix}
=0,
\end{equation*}
and therefore the backward Kreiss-Lopatinskii condition fails in the region $\{|\xi|^2+\delta<0\}$.
Note however that the Kreiss-Lopatinskii condition is true for the backward Dirichlet boundary value problem. 
It is also true for the Neuman boundary value problem provided we choose 
$(b_1,b_2)=(0,1/\text{sq}(|\xi|^2-i\tau))$ instead of $(b_1,b_2)=(0,1/\sqrt{|\xi|^2-i\tau})$.
The fact that the BVP with transparent boundary condition is not reversible is rather natural: the dissipation 
due to waves going out of the domain prevents to go back in time.
\paragraph{Well-posedness} The main result of this section states that theorem \ref{thtrivial} is true in the case 
of the pure BVP.
\begin{prop}\label{WPBVP}
If $B$ satisfies the Kreiss-Lopatinskii condition \eqref{UKL}, and 
$g\in \mathcal{H}^s_0(\R^+)$, $0 \leq s\leq 2$ $(\mathcal{H}^{1/2}_{00}$ if $s=1/2$), 
the problem \eqref{BVP} has a unique solution. 
Moreover it satisfies\footnote{We recall our unusual notation $B^1_{p,2}:=W^{1,p}$, $B^{2}_{q,2}:=W^{2,q}$}
\begin{eqnarray}
% \text{for }s=2k,\ k\in \N,\frac{2}{p}+\frac{d}{q}=\frac{d}{2},\ p>2,\ 
% \|u\|_{L^p(\R^+,W^{2k,q})\cap W^{k,p}(\R^+,L^q)}\lesssim  \|g\|_{\mathcal{H}^s_0(\R^+)},
% \\
\label{strichartzBVP}
\text{for } 0\leq s\leq 2,\ \frac{2}{p}+\frac{d}{q}=\frac{d}{2},\ p>2,\ 
\|u\|_{L^p(\R^+,B^{s}_{q,2})\cap B^{s/2}_{p,2}(\R^+,L^q)}\lesssim  \|g\|_{\mathcal{H}^s_0(\R^+)}.
\end{eqnarray}
% The result remains true in the case of Neuman boundary condition $\partial_yu|_{y=0}=g$ 
% with the following modifications : $g\in \mathcal{H}'(\R^+)$, and the solution constructed satisfies the estimate 
% \begin{equation}\label{strichartzNeu}
% \text{for } \frac{2}{p}+\frac{d}{q}=\frac{d}{2},\ p>2,\ 
% \|u\|_{L^p(\R^+,L^q)\cap L^p(\R^+,L^q)}\lesssim  \|g\|_{\mathcal{H}'(\R^+)}.
% \end{equation}
\end{prop}

\begin{proof}
\textbf{Existence} We first justify the existence of $g_n$ as in definition \ref{defsol2}. 
For any $M>0$, according to corollary \ref{normeloc} there exists 
$g_M\in \mathcal{H}^s(\R)$ that coincides with $g$ for $t\in [0,M]$, and vanishes if $t\leq 0$ or 
$t\geq M+1$. Next we shift $g_M^\delta (x,t)=g_M(x,t-\delta)$, and recall 
$\lim_0 \|g_M^\delta-g_M\|_{\mathcal{H}^s}=0$.
Let $\rho\in C_c^\infty(\R^{d-1}\times\R_t)$ with $\text{supp}(\rho)\subset 
\{|t|\leq 1\}$, $\int \rho\,dx dt=1$. Then 
setting $\rho_{\varepsilon}=\rho(\cdot/\varepsilon)/\varepsilon^d$,
\begin{eqnarray*}
\|\rho_\varepsilon*g_M^\delta-g_M^\delta\|_{\mathcal{H}^s}^2=\iint_{\R^{d-1}\times \R}
|1-\widehat{\rho}(\varepsilon\xi,\varepsilon\delta)|^2|\widehat{g_M^\delta}|^2(1+|\xi|^2+|\delta|)^s
\sqrt{||\xi|^2+\delta|}d\delta d\xi\rightarrow_{\varepsilon}0,\\
\text{supp}(\rho_\varepsilon*g_M^\delta)\subset\{(x,t):\ \delta-\varepsilon\leq t\leq M+1+\delta+\varepsilon\}.
\end{eqnarray*}
Now we remark $g_M^\delta\in \mathcal{H}^s\Rightarrow e^{-it\Delta}g\in \dot{H}^{1/4}(\R,L^2)\subset L^4(\R,L^2)$, 
thus $\rho_\varepsilon *g_M^\delta\in \cap_{k\geq 0}H^k$. Moreover 
$\lim_{M\rightarrow \infty}\|g\|_{\mathcal{H}^s([M,\infty[)}=0$ thus if $P_0g$ is the extension by zero for $t\leq 0$
\begin{equation*}
\lim_{M\rightarrow \infty}\lim_{\delta\to 0}\lim_{\varepsilon\rightarrow 0}\|\rho_\varepsilon*g_M^\delta
-P_0g\|_{\mathcal{H}^s(\R^{d-1}\times\R_t)}=0.
\end{equation*}
We also remark that for $\varepsilon\leq \delta,\ \text{supp}(\rho_\varepsilon *g_M^\delta)\subset \{t\geq 0\}$, 
so that an appropriate choice of $\varepsilon_n\leq \delta_n,M_n,$ provides a smooth sequence $(g_n)$ as in 
definition \ref{defsol2}.\\
For such $g_n$, we postpone the existence of a smooth solution $u_n$ and a priori estimate 
\eqref{strichartzBVP} to the next paragraphs. 
Now if \eqref{strichartzBVP} is true \emph{for smooth solutions}, the case 
$(p,q)=(\infty,2)$ implies that $(u_n)$ converges to a solution $u$ in $L^\infty H^s$, and the estimate 
on $u_n$ for general $(p,q)$ provides the estimate on $u$.\vspace{1mm}\\
\textbf{Uniqueness } It is again a consequence of the a priori estimate applied to the smooth solutions.\vspace{2mm}
The main issue is thus to prove estimate \eqref{strichartzBVP} : it was obtained very recently by \cite{RSZ} with 
$\|u\|_{L^pW^{s,q}}$ in the left hand side for \emph{bounded} time intervals. 
While the core of the $L^p_tL^q$ estimate does not require significant modifications we include a full proof for 
comfort of the reader.\\
% The case of Neuman boundary conditions requires a bit of care and will be treated separately
% \textcolor{red}{faire et ref \`a mettre}.
\end{proof}

\paragraph{Proof of estimate \eqref{strichartzBVP}, $s=0$} We assume that the Kreiss-Lopatinskii condition \eqref{UKL}
is satisfied, and that $g\in \cap_{k\geq 0} H^k_0(\R^{d-1}\times \R_t^+)$.
Let us look back at the formal computation leading to \eqref{solEDO}, which reads
\begin{equation*}
\mL u=e^{-\sqrt{|\xi|^2-i\tau}\,y}\frac{\mL g}{b\cdot V_-(\xi,\tau)}.
\end{equation*}
According to the Kreiss-Lopatinskii condition \eqref{UKL}, $|\mathcal{L}g/(b\cdot V_-)|\sim |\mathcal{L}g|$, uniformly 
in $(\tau,\xi)$, so that using Paley-Wiener's theorem $\mathcal{L}g/(b\cdot V_-)$ is the Fourier-Laplace transform 
of some $g_1$ supported in $t\geq 0$. \\
% independently of $\delta,\xi$. As all computations that follow only rely on pointwise estimates for 
% $\frac{\mL g}{b\cdot V_-(\xi,i\delta)}$ we only carry the proof for the case  
% $\frac{\mL g}{b\cdot V_-(\xi,i\delta)}=\mL g$ 
% (Dirichlet conditions), the general case is exactly similar\footnote{Obviously, this is not true for Neuman boundary 
% conditions}. 
Now we let $\gamma\rightarrow 0$: since $g\in \cap_{k\geq 0}H^k_0(\R^{d-1}\times \R_t^+)$, its zero extension 
belongs to $\cap_{k\geq 0}H^k(\R^{d-1}\times \R_t)$, and we can (abusively) identify 
$\mL g(\xi,i\delta)=\widehat{P_0g}(\xi,\delta)$, with 
\begin{equation*}
\forall\,k\geq 0,\ \int_{\R^{d-1}\times \R}(1+|\xi|^2+|\delta|^2)^k|\widehat{P_0g}|^2d\delta d\xi<\infty. 
\end{equation*}
Since $|\widehat{g_1}|\sim |\widehat{P_0 g}|$, $g_1\in \cap_{k\geq 0}H^k(\R^{d-1}\times \R_t)$ 
and for any $s\geq 0$, $\|g_1\|_{\mathcal{H}^s}\sim \|g\|_{\mathcal{H}^s}$, moreover it is supported in $t\geq 0$ thus 
its restriction belongs to $\cap_{k\geq 0} \mathcal{H}^s_0(\R_t^+)$. We are reduced 
to solve the IBVP \eqref{BVP} with smooth Dirichlet boundary condition $g_1$. 
We abusively denote $\widehat{u}(\xi,\delta)$ for $\mathcal{L}u(\xi,i\delta)$, drop the index $1$ of $g_1$ 
and simply assume 
\begin{equation*}
\widehat{u}(\xi,\delta)=e^{-\sqrt{|\xi|^2+\delta}\,y}\widehat{g}(\xi,\delta),\ g\in \cap_{k\geq 0}H^k_0.
\end{equation*}
If $\delta +|\xi|^2\geq 0$, $\sqrt{\delta+|\xi|^2}
\in \R^+$is the usual square root, else $\sqrt{\delta+|\xi|^2}=-i\sqrt{|\delta+|\xi|^2|}$. 
The solution $u(x,y,t)$ is then obtained by inverse Fourier transform. We split the integral depending on the sign 
of $\delta+|\xi|^2$, the change of variables $\delta+|\xi|^2=\pm\eta^2$ gives 
\begin{eqnarray}
\nonumber u(x,y,t)
&=&\frac{1}{(2\pi)^{d}}\int_{\R^{d-1}}\int_{\delta\leq -|\xi|^2} e^{iy\sqrt{|\delta+|\xi|^2|}}
e^{i(\delta t+x\cdot\xi)}\widehat{g}(\xi,\delta)d\delta d\xi\\
\nonumber 
&&+\frac{1}{(2\pi)^{d}}\int_{\R^{d-1}}\int_{\delta>-|\xi|^2} e^{-y\sqrt{\delta+|\xi|^2}}e^{i(\delta t+x\cdot\xi)}
\widehat{g}(\xi,\delta)d\delta d\xi
\\ 
\nonumber 
&=& \frac{1}{(2\pi)^{d}}\int_{\R^{d-1}}\int_{0}^\infty e^{i(y\eta+x\cdot\xi)}e^{-it(|\xi|^2+\eta^2)}
2\eta \widehat{g}(\xi,-\eta^2-|\xi|^2)d\eta\, d\xi\\
\nonumber 
&&+\frac{1}{(2\pi)^{d}}\int_{\R^{d-1}}\int_{0}^\infty e^{-y\eta+ix\cdot\xi}e^{it(-|\xi|^2+\eta^2)}
2\eta\widehat{g}(\xi,-|\xi|^2+\eta^2)d\eta\, d\xi\\
\label{formBVP}&:=& u_1+u_2.
\end{eqnarray}

% 
% \begin{prop}\label{limibvp}
% Dans le cas Dirichlet $\lim_{-\infty}\|u(t)\|_{L^2}=0$. En effet, pour tout $T$, $u|_{t\leq T}$ ne d\'epend 
% que de $g|_{t\leq T}$ (cf causalit\'e de l'article CPDE), mais \'etant donn\'e une fonction de troncature 
% $\|\psi(t+T)g\|_{\mathcal{H}}\leq \|\psi(t+T)g\|_{H^1}\rightarrow_{T\rightarrow \infty} 0$, on conclut par densit\'e.
% \end{prop}
\noindent From the smoothness of $g$ the formula is absolutely convergent, infinitely differentiable 
in $x,y,t$, and clearly gives a solution to \eqref{BVP}, so that the formal computation is justified for smooth solutions. 
Moreover, the formulas are well defined for $t\in \R$ (and 
actually cancels for $t<0$ by Paley-Wiener's theorem), therefore we will focus on proving the seemingly stronger, 
but more natural estimate 
\begin{equation}\label{improvedStric}
\|u\|_{L^p(\R,L^q)}\lesssim \|g\|_{\mathcal{H}(\R^+)}.
\end{equation}

\paragraph{Control of $u_1$} Let $\widehat{\phi}(\xi,\eta):=2\eta\widehat{g}
(\xi,-\eta^2-|\xi|^2)1_{\eta\geq 0}$, we observe $u_1(x,y,t)=e^{it\Delta} \phi$, 
so that the classical Strichartz estimate \eqref{strichartzCauchy} gives 
\begin{eqnarray}
\nonumber
\|u_1\|_{L^p(\R^+,L^q(\R^{d-1}\times\R^+)}\leq \|u_1\|_{L^p(\R,L^q(\R^d))}\lesssim 
\|\phi\|_{L^2}
&\sim& \|\widehat{\phi}\|_{L^2}\\
&\sim&\iint \eta^2|\widehat{g}(\xi,-|\xi|^2
-\eta^2)|^2d\eta d\xi\\
\label{estimu1}&\sim& \int_{\R^d}\int_{-\infty}^{-|\xi|^2}\sqrt{||\xi|^2+\delta|}|\widehat{g}(\xi,\delta)|^2d\delta d\xi
\\
&\leq& \|g\|_{\mathcal{H}}^2.
\end{eqnarray}
\paragraph{Control of $u_2$} As mentioned before, it is more convenient to let $t$ vary in $\R$ rather than $\R^+$, 
obviously bounds in $L^p(\R,L^q(\R^{d-1}\times\R^+))$ imply 
bounds in $L^p(\R^+,L^q(\R^{d-1}\times\R^+))$.\\
The idea in \cite{RSZ} is to use a $TT^*$ argument similar to the classical one for the Schr\"odinger
equation, namely if we set $\widehat{\psi}=2\eta \widehat{g}(\xi,-|\xi|^2+\eta^2)1_{\eta\geq 0}$ then 
\eqref{formBVP} reads 
\begin{equation*}
u_2=\frac{1}{(2\pi)^{d}}\int_{\R^{d-1}}\int_{\R}e^{-y|\eta|+ix\cdot\xi}e^{it(-|\xi|^2+\eta^2)}
\widehat{\psi}(\xi,\eta)d\eta\, d\xi :=T(\psi),
\end{equation*}
with $\|\psi|\|_{L^2}\lesssim \|g\|_{\mathcal{H}}$.
Consider $T$ as an operator $L^2(\R^{d-1}\times \R)\to L^p(\R_t,L^q(\R^{d-1}\times \R^+))$, the $TT^*$ argument 
consists in proving 
\begin{equation*}
\|TT^*\|_{L^{p'}L^{q'}\to L^pL^q}<\infty.  
\end{equation*}
If such a bound holds true, then $\| T^*f\|_{L^2}^2=\langle TT^*f,f\rangle \lesssim \|f\|_{L^{p'}L^{q'}}^2$, thus 
$T^*$ is continuous $L^{p'}L^{q'}\to L^2$, and by duality 
$T:L^2\to L^pL^q$ is continuous, which gives the expected bound $\|u_2\|_{L^pL^q}\lesssim \|g\|_{\mathcal{H}}$.
Now let us write
\begin{eqnarray*}
u_2(x,y,t)&=&\frac{1}{(2\pi)^{d}}\iint_{\R^{d}} \iint_{\R^{d}}
e^{-y|\eta|-it(|\xi|^2-\eta^2)+ix\cdot\xi}e^{-ix_1\cdot\xi-iy_1\eta}
\psi(x_1,y_1)dx_1dy_1d\eta d\xi \\
&=& \frac{1}{(2\pi)^{d}}\iint_{\R^{d}}\bigg(\iint_{\R^{d}}e^{-y|\eta|-it(|\xi|^2-\eta^2)+ix\cdot\xi}
e^{-ix_1\cdot\xi-iy_1\eta}d\xi d\eta\bigg)\psi(x_1,y_1)dx_1dy_1.
\end{eqnarray*}
We denote\footnote{While $X$ corresponds to the space variable $(x,y)$ that we use throughout the paper, the variable $X_1$ is 
purely artificial.}
$X=(x,y)\in \R^{d-1}\times \R^+$, $X_1=(x_1,y_1)\in \R^d$, observe that $T\psi$ can be seen as the action of a kernel 
with parameter
$K_t(X,X_1)$ on $\psi(X_1)$:
\begin{equation*}
u_2(x,y,t)=\frac{1}{(2\pi)^d}Op(K_t)\cdot \psi. 
\end{equation*}
 According to the $TT^*$ argument, 
it suffices to bound $Op(K_t)\circ Op(K_t)^*:\ L^{p'}(\R,L^{q'}(\R^{d-1}\times \R^+))\to 
L^p(\R,L^q(\R^{d-1}\times \R^+))$. After a few computations one may check
\begin{eqnarray}
\nonumber
Op(K_t)\circ Op(K_t)^*\,f&=&\int_{\R^{d-1}\times \R^+\times \R_s}\bigg(\int_{\R^d}K_t(X,X_1)
\overline{K_s}(X_2,X_1)dX_1\bigg)f(X_2,s)dX_2ds
\\
\label{defkt}
&:=&\int_{\R_s} \bigg(\int_{\R^{d-1}\times \R^+}N_{t,s}(X,X_2)f(X_2,s)dX_2\bigg)ds\\
&=&\int_{\R_s} \big(\text{Op}(N_{t,s})\cdot f(\cdot,s)\big)(X)ds
\end{eqnarray}
% We freeze $t,s$ and consider equation \eqref{defkt} as the integral over $s\in \R$ of the action of a Kernel $N_{t,s}$ on 
% $f(\cdot,s)$.
% , 
% for which we prove two preliminary facts:
% \begin{enumerate}
%  \item $\sup_{X,X_2}|N_{t,s}(X,X_2)|\lesssim 1/|t-s|^{d/2}$
%  \item $\text{Op}(N_{t,s})$ is bounded $L^2\to L^2$ uniformly in $(t,s)\in \R^2$.
% \end{enumerate}

\begin{lemma}\label{idNts}
We have for $(X,X_2)\in (\R^{d-1}\times\R^+)^2$
\begin{equation*}
N_{t,s}(X,X_2)=(2\pi)^d\int_{\R^d}e^{i(|\xi|^2-\eta^2)(s-t)}e^{i\xi\cdot(x-x_2)}e^{-(y+y_2)|\eta|}d\eta d\xi.
\end{equation*}
\end{lemma}

\begin{proof}
According to identity \eqref{defkt}
\begin{eqnarray*}
N_{t,s}&=&\int_{\R^d} K_t(X,X_1)\overline{K_s}(X_2,X_1)dX_1\\
&=&\int_{\R^{3d}} e^{-it(|\xi|^2-\eta^2)+ix\cdot\xi-y|\eta|-i(\xi\cdot x_1
+\eta y_1)}e^{is(|\xi_1|^2-\eta_1^2)-ix_2\cdot\xi_1-y_2|\eta_1|+i(\xi_1\cdot x_1+\eta_1y_1)}\\
&& \hspace{10cm}d\eta d\xi d\eta_1 d\xi_1 dX_1
\\
&=&(2\pi)^d\int_{\R^d}e^{-it(|\xi|^2-\eta^2)+ix\cdot\xi-y|\eta|}\mathcal{F}_{X_1\rightarrow \xi,\eta}
\mathcal{F}^{-1}_{(\xi_1,\eta_1)\rightarrow X_1}\big(e^{-ix_2\cdot\xi_1
-y_2|\eta_1|+is(|\xi_1|^2-\eta_1^2)}\big)d\xi d\eta\\
&=& (2\pi)^d\int_{\R^d}e^{-it(|\xi|^2-\eta^2)+ix\cdot\xi-y|\eta|}\,e^{-ix_2\xi
-y_2|\eta|+is(|\xi|^2-\eta^2)}d\xi d\eta,
\end{eqnarray*}
which is the expected result.
\end{proof}

The estimate of $N_{t,s}$ requires a (classical) substitute to Plancherel's formula :

\begin{lemma}\label{planchdupauvre}
The map $\mL:f\rightarrow \int_0^\infty e^{-\lambda y}f(y)dy$ is continuous $L^2(\R^+)\rightarrow L^2(\R^+)$.
\end{lemma}

\begin{proof}
 We have
\begin{equation*}
\|\mL f\|_2^2=\langle \mL f,\mL f\rangle =\int_{(\R^+)^3}e^{-\lambda(y_1+y_2)}f(y_1)\overline{f}(y_2)dy_2dy_1d\lambda
=\int_{(\R^+)^2}\frac{f(y_1)\overline{f}(y_2)}{y_1+y_2}dy_1dy_2
\end{equation*}
Splitting $(\R^+)^2=\{y_2\leq y_1\}\cup\{y_1\leq y_2\}$, we remark
\begin{equation*}
\|\mL f\|_2^2=\int_{0}^\infty f(y_1)\frac{1}{y_1}\int_0^{y_1}\frac{\overline{f}(y_2)}{1+y_2/y_1}dy_2dy_1
+\int_{0}^\infty \overline{f}(y_2)\frac{1}{y_2}\int_0^{y_2}\frac{f(y_1)}{1+y_1/y_2}dy_1dy_2.
\end{equation*}
One easily concludes using $|f(y_2)/(1+y_2/y_1)|\leq |f(y_2)|$ and Hardy's inequality.
\end{proof}

\begin{prop}
The operator $\text{Op}(N_{t,s})$ satisfies for $2\leq p\leq \infty$ 
\begin{eqnarray}
\|\text{Op}(N_{t,s})v\|_{L^p(\R^{d-1}\times \R^+)} \lesssim \frac{\|v\|_{L^{p'}(\R^{d-1}\times \R^+)}}
{|t-s|^{d(1/2-1/p)}}.
\end{eqnarray}
\end{prop}

\begin{proof}
The case $p=\infty$: according to proposition \ref{idNts} 
\begin{eqnarray*}
N_{t,s}(X,X_2)&=&\int_{\R^d}e^{i(|\xi|^2-\eta^2)(s-t)}e^{i\xi\cdot (x-x_2)}e^{-(y+y_2)|\eta|}d\eta d\xi\\
&=& \int_{\R^{d-1}}e^{i|\xi|^2(s-t)}e^{i\xi\cdot (x-x_2)}d\xi \int_{\R}e^{i\eta^2(t-s)}e^{-(y+y_2)|\eta|}d\eta\\
&=& \frac{e^{i\frac{|x-x_2|^2}{4(t-s)}}}{(4i\pi(t-s))^{(d-1)/2}} \int_{\R}e^{i\eta^2(t-s)}e^{-(y+y_2)|\eta|}d\eta
\end{eqnarray*}
The Van Der Corput lemma implies 
\begin{equation*}
\big|\int_{\R}e^{i\eta^2(t-s)}e^{-(y+y_2)|\eta|}d\eta \big|\lesssim \frac{\|e^{-(y+y_2)|\eta|}\|_{L^\infty_\eta}+
\|(e^{-(y+y_2)|\eta|})'\|_{L^1_\eta}}{\sqrt{|t-s|}}\lesssim \frac{1}{|t-s|^{1/2}}.
\end{equation*}
Therefore $|N_{t,s}|\lesssim 1/|t-s|^{d/2}$ uniformly in X,$X_2$, this implies the case $p=\infty$.\\
For the case $p=2$ we use Plancherel's formula and lemma \ref{planchdupauvre}: 
\begin{eqnarray*}
\|\text{Op}(N_{t,s})v\|_{L^2}&=&\bigg\|\int_{\R^{d-1}}e^{i|\xi|^2(s-t)}e^{i\xi\cdot x}
\int_{\R\times \R^+\R^{d-1}} e^{-i\xi\cdot x_2-i\eta^2(s-t)-(y+y_2)\eta}v(X_2)d\eta dX_2 d\xi\bigg\|_{L^2_{xy}}
\\
&\sim&\bigg\|\bigg\|\int_{\R^{d-1}}e^{-i\xi\cdot x_2} \int_{\R^{2}}
e^{-i\eta^2(s-t)}e^{-(y+y_2)|\eta|}v(X_2)d\eta dy_2 dx_2 \bigg\|_{L^2_{\xi}}\bigg\|_{L^2_y}
\\
&\sim& \bigg\|\bigg\| \int_{\R}e^{-|\eta|y} \int_{\R}
e^{-i\eta^2(s-t)}e^{-y_2|\eta|}v(X_2)d\eta dy_2 \bigg\|_{L^2_{y}}\bigg\|_{L^2_{x_2}}
\\
&\lesssim& \bigg\|\bigg\| e^{-i\eta^2(s-t)}\int_{\R}e^{-y_2|\eta|}v(X_2)dy_2 \bigg\|_{L^2_{\eta}}\bigg\|_{L^2_{x_2}}
\\
&\lesssim& \| v(X_2)\|_{L^2_{X_2}}.
\end{eqnarray*}
The general case follows from an interpolation argument.
\end{proof}

\noindent
The estimate on $\text{Op}(K_t)\circ \text{Op}(K_t)^*$ now follows from the Hardy-Littlewood-Sobolev lemma 
(e.g. theorem 2.6 in \cite{linponce}):
for $\displaystyle p>2,\ \frac{2}{p}+\frac{d}{q}=\frac{d}{2}$, we have 
$\displaystyle 1+\frac{1}{p}=\frac{1}{p'}+d\bigg(\frac{1}{2}-\frac{1}{q}\bigg)$, thus
\begin{eqnarray*}
\|\text{Op}(K_t)\circ \text{Op}(K_t)^*f\|_{L^p_tL^q_X}&=&
\bigg\|\int_{\R} \text{Op}(N_{t,s})f(\cdot,s)ds\bigg\|_{L^p_tL^q_X}
\\
&\lesssim& \bigg\|\int_{\R} \frac{\|f(\cdot,s)\|_{L^{q'}_X}}{|t-s|^{d(1/2-1/q)}}ds\bigg\|_{L^p_t}
\\
&\lesssim& \|f\|_{L^{p'}_tL^{q'}_X}.
\end{eqnarray*}
Using the $TT^*$ argument this ends estimate \eqref{strichartzBVP} for the case $s=0$.
\paragraph{Estimate \eqref{strichartzBVP}, the case $s=2$}
By differentiation of formula \eqref{formBVP}, for $|\alpha|+\beta+2\gamma\leq 2$, and using the case $s=0$
\begin{equation*}
\|\partial_x^\alpha\partial_y^\beta\partial_t^\gamma u\|_{L^p_tL^q}\lesssim 
\bigg(\iint ||\xi|^2+\delta|^{(1+\beta)/2}|\xi|^\alpha |\delta|^\gamma
|\widehat{g}|^2
d\delta d\xi\bigg)^{1/2}\lesssim \|g\|_{\mathcal{H}^2_0(\R_t^+)}^2.
\end{equation*}
\begin{rmq}
We recall that in the inequality above, $\widehat{g}$ is the Fourier transform of the extension of $g$
that vanishes for $t\leq 0$, which is why we can not simply take $g\in \mathcal{H}^2(\R^+)$.\\
Obviously, the same argument applies as soon as $s$ is an even integer, but since the non-integer case is slightly more 
delicate, we chose to consider only $s\leq 2$ for simplicity.
\end{rmq}
% Remarquons aussi $||\xi|^2+\delta|+|\xi|^2\geq (|\xi|^2+|\delta|)/2$, ce qui donne une norm\'e \'equivalente plus simple
% et nous am\`ene la d\'efinition suivante:
% \begin{define}
% L'espace $\mathcal{H}^k$ est d\'efini par 
% \begin{equation*}
% \|g\|_{\mathcal{H}^k}=\bigg(\iint ||\xi|^2+\delta|^{1/2}(|\xi|^{2k}+|\delta|^k)|\widehat{g}|^2
% d\delta d\xi\bigg)^{1/2}.
% \end{equation*}
% \end{define}
% Enfin il faut aussi de la r\'egularit\'e en temps.
\paragraph{Estimate \eqref{strichartzBVP}, the case $0<s<2$} 
% This is an interpolation argument, but since we work 
% on $t\geq 0$, a bit of care is required. First we recall that the estimates are true for $t\in \R$ 
% if $g\in \mathcal{H}^2_0$, then 
% the solution constructed can be smoothly extended by $0$ for $t\leq 0$, indeed the formula for the explicit 
% solution in \eqref{formBVP} defines a function in $C(\R,H^2)$, and by Paley-Wiener's theorem it vanishes on $t\leq 0$.
% \\
This is an interpolation argument. For $p>2,\ 2/p+d/q=d/2$, the solution map is continuous 
$$\mathcal{H}(\R^+)\rightarrow L^p(\R_t,L^q),\ \mathcal{H}^2_0(\R^+)\to L^p(\R_t,W^{2,q})
\cap W^{1,p}(\R_t,L^q),$$ 
thus by interpolation it is continuous 
$[\mathcal{H},\mathcal{H}^2_0]_{s,2}\to L^p(\R_t,B^{2s}_{q,2})\cap B^s_{p,2}(\R_t,L^q)$, this gives the 
result by using the interpolation identities of proposition \ref{interpHs0} and by restriction on $t\geq 0$. 
% Note that the $\mathcal{H}^{1/2}_{00}$ topology differs from the $\mathcal{H}^{1/2}$ one as it is given by 
% the norm $\|g\|_{\mathcal{H}^{1/2}(\R^+)}+I(g)^{1/2}$.
\begin{flushright}
$\displaystyle \qed $
\end{flushright}

% \paragraph{The case of Neuman boundary conditions} We only sketch the proof, since most of it is exactly similar to 
% the Dirichlet case. Consider the problem  
% \begin{equation}\label{pbaux}
% \left\{
% \begin{array}{ll}
% i\partial_t u+\Delta u=0,\ (x,y,t)\in \R\times \R^+\times \R\\
% \partial_yu|_{y=0}=g\in \mathcal{H}',\\
% u(\cdot,0)=0.
% \end{array}
% \right.
% \end{equation}
% By a similar density argument, we can assume that $g$ is smooth, the computations leading to 
% \eqref{formBVP} give here
% \begin{eqnarray*}
% u(x,y,t)
% &=&\frac{1}{(2\pi)^{d}}\int_{\R^{d}} e^{-y\sqrt{\delta+|\xi|^2}}
% e^{i(\delta t+x\cdot\xi)}\frac{\widehat{g}(\xi,\delta)}{-\sqrt{\delta+|\xi|^2}}d\delta d\xi,
% \end{eqnarray*}
% We recall that $\mathcal{H}'=\{g_1:\ \iint |\widehat{g_1}(\xi,\delta)|^2/\sqrt{||\xi|^2+\delta|}d\xi 
% d\delta <\infty\}$. According to estimate \eqref{strichartzBVP} with $s=0$, we have 
% \begin{equation*}
% \|u\|_{L^pL^q}^2\lesssim \|u|_{y=0}\|_{\mathcal{H}}^2=\iint_{\R^d} \sqrt{||\xi|^2+\delta|}
% |\widehat{u|_{y=0}}|^2d\xi d\delta
% =\int_{\R^d} \frac{|\widehat{g}(\xi,\delta)|^2}{\sqrt{||\xi|^2+\delta|}}d\xi d\delta=\|g\|_{\mathcal{H}'}^2.
% \end{equation*}
% which is precisely the expected estimate.
% 

\paragraph{The boundary value problems on $[-T,\infty[$ and $\R_t$} A natural question (and actually useful in the rest 
of the paper) is the solvability
of the BVP on other time intervals than $[0,\infty[$. As we mentioned before, the backward BVP can be ill-posed. 
However translations have a better behaviour: first, we extend the operator 
$(a,b)\to B(a,b)$ to distributions in $\mathcal{H}(\R)\times \mathcal{H}'(\R)$ with the formula
\begin{equation*}
B(a,b)=\mathcal{F}_{x,t}^{-1}\big(b_1(\xi,i\delta)\widehat{a}(\xi,\delta)+b_2(\xi,i\delta)\widehat{b}(\xi,\delta)\big).
\end{equation*}
Under the Kreiss-Lopatinskii condition, this extension maps $\mathcal{H}\times \mathcal{H}'\to \mathcal{H}(\R)$.
For $g\in \mathcal{H}(\R)$ smooth, supported in $t\geq 0$ and $u$ a smooth solution to the pure BVP \eqref{BVP}, 
we define $u_T=u(t+T)$ for some $T\in \R$. Then from the explicit formula \eqref{solEDO}, $u_T$ satisfies 
\begin{equation*}
\mathcal{F}B(u_T|_{y=0},\ \partial_yu_T|_{y=0})=e^{-iT\delta}\mathcal{L}g(\xi,i\delta)=\mathcal{F}(g(\cdot+T)),
\end{equation*}
so that $u_T$ is a solution of the BVP 
 \begin{equation*}
 \left\{
 \begin{array}{ll}
  i\partial_tv+\Delta v=0, \ (x,y,t)\in \R^{d-1}\times \R^+\times [-T,\infty[, \\
  B(v|_{y=0},\partial_{y}v|_{y=0})=g(\cdot+T),\\
  v(\cdot,-T)=0.
 \end{array}
\right.
\end{equation*}
Therefore up to the appropriate translation of $g$, to solve a BVP on $[-T,\infty[$ is equivalent to solve a BVP 
on $[0,\infty[$. 
A useful consequence of this remark is the well-posedness of the BVP 
posed on $\R^{d-1}\times \R^+\times \R_t$. 

\begin{coro}\label{BVPR}
Consider the boundary value problem 
 \begin{equation}
 \left\{
 \begin{array}{ll}
  i\partial_tu+\Delta u=0,\\
  B(u|_{y=0},\partial_{y}u|_{y=0})=g,\\
  \lim_{t\rightarrow -\infty}u(\cdot,t)=0.
 \end{array}
\right.(x,y,t)\in \R^{d-1}\times \R^+\times \R_t. 
\end{equation}
If $B$ satisfies the Kreiss-Lopatinskii condition \eqref{UKL} and $g\in \mathcal{H}^s(\R)$, $0\leq s\leq 2$, 
there exists a unique 
solution $u\in C(\R,H^s)$, moreover it satisfies estimate \eqref{strichartzBVP} with $\R_t^+$ replaced by $\R_t$.\\
If $g$ vanishes on $\R^{d-1}\times ]-\infty,T]$, then so does $u$ on $(\R^{d-1}\times \R^+)\times ]-\infty,T]$.
% \\
% If $B$ is the Neuman boundary condition and $g\in \mathcal{H}'(\R)$, there exists a unique solution $u\in C(\R,L^2)$, 
% moreover it satisfies estimate \eqref{strichartzNeu} with $\R$ instead of $\R^+$.
\end{coro}

\begin{proof}
Fix $g\in \mathcal{H}^s(\R)$. By density there exists $g_n\in C_c^\infty(\R^{d-1}\times \R_t)$ such that 
\begin{equation*}
\|g-g_n\|_{\mathcal{H}^s(\R_t)}\longrightarrow_n 0.
\end{equation*}
We can assume that $g_n$ is supported in $[-T_n,\infty[$, and $T_n$ is increasing. By translation invariance in time, 
there exists a smooth solution $u_n$ to 
 \begin{equation}\label{bvptranslate}
 \left\{
 \begin{array}{ll}
  i\partial_tu_n+\Delta u_n=0,\\
  B(u_n|_{y=0},\partial_{y}u_n|_{y=0})=g_n,\\
  u_n(\cdot,-T_n)=0.
 \end{array}
\right.(x,y,t)\in \R^{d-1}\times \R^+\times [-T_n,\infty[. 
\end{equation}
As was pointed out in the proof of estimate \eqref{strichartzBVP}, setting $u_n|_{]-\infty,-T_n[}=0$ defines a 
smooth extension of $u_n$, which solves the boundary value problem with $g_n|_{]-\infty,-T_n]}=0$.\\
Let $n\geq p$, then $\text{supp}(u_n-u_p)\subset [-T_n,\infty[$ and a priori 
estimate \eqref{strichartzBVP} implies
\begin{equation*}
\|u_n-u_p\|_{L^\infty(\R, H^s(\R))}\lesssim \|g_n-g_p\|_{\mathcal{H}^s([-T_n,\infty[)}\lesssim 
\|g_n-g_p\|_{\mathcal{H}^s(\R)}.
\end{equation*}
This implies that $(u_n)$ converges to some $u\in C_tH^s$. Moreover 
$$\forall\,n\in \N,\ \displaystyle \lim_{-\infty}\|u_n(t)\|_{H^s}=0\Rightarrow 
\lim_{-\infty}\|u(t)\|_{H^s}=0.$$
The other estimates can be obtained as for proposition \ref{WPBVP}.\\
In the case where $g$ is supported in $\R^{d-1}\times [T,\infty[$, it suffices to observe that we can assume that 
$g_n$ is supported in $\R^{d-1}\times [T_n,\infty[$, and use the previous observation on the support of smooth solutions.
\end{proof}

\subsection{Estimates for the Cauchy problem}\label{secCauchy}

\paragraph{Pure Cauchy problem} 
We recall (see \eqref{UKL}) that the Kreiss-Lopatinskii condition reads 
$\alpha\leq |b_1-\sqrt{|\xi|^2+\delta}b_2|\leq \beta$, therefore we define 
$\Lambda$ the Fourier multiplier of symbol $\sqrt{||\xi|^2+\delta|	}$ that 
acts on functions defined on $\R^{d-1}\times \R_t$. In order to control 
$\|B(u|_{y=0},\partial_yu|_{y=0})\|_{\mathcal{H}^s}$ we need to control 
$\|u|_{y=0}\|_{\mathcal{H}^s}$ and $\|\Lambda^{-1}\partial_yu|_{y=0}\|_{\mathcal{H}^s}$.
\begin{prop}\label{estimcauchy}
The solution $e^{it\Delta}u_0$ of the Cauchy problem 
\begin{equation*}
\left\{
\begin{array}{ll}
i\partial_tu+\Delta u=0,\\
u|_{t=0}=u_0,
\end{array}
\right. (x,y,t)\in \R^{d+1},
\end{equation*}
satisfies the following estimates for $0\leq s\leq 2 $:
\begin{eqnarray}
\label{stricCauchy}
\forall\,p>2,\ \frac{2}{p}+\frac{d}{q}=\frac{d}{2},\ 
\|u\|_{L^p(\R_t,B^s_{q,2})\cap B^{s/2}_{p,2}(\R_t,L^q) }\lesssim \|u_0\|_{H^s} ,\\
\label{traceCauchy}
\|u|_{y=0}\|_{\mathcal{H}^s(\R_t)}
+\|\Lambda^{-1}(\partial_yu|_{y=0})\|_{\mathcal{H}^s(\R_t)}\lesssim \|u_0\|_{H^s}.
\end{eqnarray}

\end{prop}
\begin{proof}
The $L^pB^s_{q,2}$ estimate in \eqref{stricCauchy} is the classical Strichartz estimate,
see e.g. \cite{Cazenave} Corollary 2.3.9. Since $\partial_tu=i\Delta u$, 
$\|u\|_{W^{1,p}L^q}\lesssim \|u_0\|_{H^{2}}$, and the $B^{s/2}_{p,2}L^q$ bound follows 
by interpolation. For the trace estimate, we observe that the solution of the Cauchy problem satisfies
\begin{eqnarray*}
\forall\,(x,y,t)\in \R^{d+1},\
(e^{it\Delta}u_0)(x,y)=\frac{1}{(2\pi)^d}\iint e^{-i(|\xi|^2+\eta^2)t}e^{ix\cdot \xi+iy\eta}
\widehat{u_0}(\xi,\eta)d\xi d\eta,\\
\Rightarrow (e^{it\Delta}u_0)(x,0)=\frac{1}{(2\pi)^d}\iint e^{-i(|\xi|^2+\eta^2)t}e^{ix\cdot \xi}
\widehat{u_0}(\xi,\eta)d\xi d\eta.
\end{eqnarray*}
We consider the integral over $\eta\geq 0$, and use the change of variables $\delta=-(\eta^2+|\xi|^2)$
\begin{eqnarray*}
\int_{\R^{d-1}}\int_{\R^+}e^{-i(|\xi|^2+\eta^2)t}e^{ix\cdot \xi}\widehat{u_0}(\xi,\eta)d\eta d\xi 
&=&
\int_{\R^{d-1}}\int_{-\infty}^{-|\xi|^2}e^{i\delta t}e^{ix\cdot \xi}\frac{\widehat{u_0}\big(\xi,\sqrt{||\xi|^2+\delta|}}
{\sqrt{||\xi|^2+\delta|}}\big) d\delta d\xi \\
&:=& (2\pi)^{d}\mathcal{F}^{-1}_{x,t}(\psi).
\end{eqnarray*}
Then for $s\geq 0$, reversing the change of variable
\begin{eqnarray*}
\|\mathcal{F}_{x,t}^{-1}(\psi)\|_{\mathcal{H}^s(\R_t)}^2&=&\int_{\R^d} \sqrt{|\delta+|\xi|^2|}(1+|\delta|+|\xi|^2)^s
|\psi(\xi,\delta)|^2d\delta d\xi \\
 &=& \int_{\R^{d-1}}\int_{-\infty}^{-|\xi|^2}\sqrt{|\delta+|\xi|^2|}(1+|\delta|+|\xi|^2)^s
 \bigg|\frac{\widehat{u_0}\big(\xi,\sqrt{||\xi|^2+\delta|}\big)}{\sqrt{||\xi|^2+\delta|}}\bigg|^2d\delta d\xi \\
 &\lesssim& \iint_{\R^{d-1}\times \R^+}(1+|\xi|^2+|\eta|^2)^s|\widehat{u_0}(\xi,\eta)|^2d\eta d\xi \sim \|u_0\|_{H^s}^2.
\end{eqnarray*}
Symmetric computations can be carried for $\eta\in \R^-$, we conclude
\begin{equation*}
\|e^{it\Delta}u_0|_{y=0}\|_{\mathcal{H}^s(\R_t)}\lesssim \|u_0\|_{H^s}.
\end{equation*}
The estimate for $\|\Lambda^{-1}(\partial_yu|_{y=0})\|_{\mathcal{H}^s}$ is done similarly by writing
\begin{equation*}
(\partial_yu|_{y=0})=
\frac{1}{(2\pi)^d}\iint_{\R^d} e^{-i(|\xi|^2+\eta^2)t}e^{ix\cdot \xi}i\eta\widehat{u_0}(\xi,\eta)d\xi d\eta,
\end{equation*}
and using the fact that after the change of variable, the $\eta$ factor becomes $\sqrt{||\xi|^2+\delta|}$, so that it 
balances precisely the symbol of $\Lambda^{-1}$.
\end{proof}
\begin{rmq}
Inequality \eqref{traceCauchy} is a multi-dimensional variant (not new) of the sharp 
Kato-smoothing property that we already mentioned in the introduction. 
It is clear that the argument actually works for $s\geq 0$.
\end{rmq}

\paragraph{Pure forcing problem}
We consider $u=\displaystyle \int_0^te^{i(t-s)\Delta}f(s)ds$ solution of 
\begin{equation*}
\left\{
\begin{array}{ll}
 i\partial_tu+\Delta u=if,\\
 u(\cdot,0)=0.
\end{array}
\right.
(x,t)\in \R^{d+1}.
\end{equation*}
Our aim is to obtain an estimate of 
the kind $\|u|_{y=0}\|_{\mathcal{H}^s(\R_t)}\lesssim \|f\|_{L^1_tH^s}$. If the integral $\int_0^t$ was replaced by 
$\int_0^\infty$, we might simply remark that 
\begin{equation*}
u|_{y=0}= e^{it\Delta}\bigg(\int_0^\infty e^{-is\Delta}f(s)ds\bigg)\bigg|_{y=0}
\end{equation*}
and use Minkowski's inequality $\big\|\int_0^\infty e^{-is\Delta}f(s)ds\big\|_{H^s}\lesssim \|f\|_{L^{1}_tH^{s}}$.
Combined with proposition \ref{estimcauchy}, this implies 
$\|\int_0^\infty e^{i(t-s)\Delta}f(s)ds|_{y=0}\|_{\mathcal{H}^s}\lesssim \|f\|_{L^{1}_tH^{s}}$.
Unfortunately, due to the intricate nature of $\mathcal{H}^s$, which measures both time and space regularity, we can 
not apply the celebrated Christ-Kiselev lemma to deduce bounds for $\int_0^te^{i(t-s)\Delta}f(s)ds|_{y=0}$ 
(see also remark \ref{rmqTimereg} for a discussion on this issue). 
Nevertheless, we have the following proposition.
\begin{prop}\label{estimDuham}
For $0<s<2$, $(p,q)$, and $(p_1,q_1)$ admissible pairs, we have 
\begin{eqnarray}
\label{stricduham}
\bigg\|\int_0^te^{i(t-s)\Delta}f(s)ds\bigg\|_{L^{p_1}(\R_t,B^s_{q_1,2})\cap B^s_{p_1,2}(\R_t,L^{q_1})}&\lesssim &
\|f\|_{L^{p'}(\R_t,B^s_{q',2})\cap B^{s/2}_{p',2}(\R_t,L^{q'})},\\
\label{traceduham}
\bigg\|\int_0^te^{i(t-\tau)\Delta}f(\tau)d\tau\bigg|_{y=0}\bigg\|_{\mathcal{H}^s(\R_t)}&\lesssim &
\|f\|_{L^{p'}(\R_t,B^s_{q',2})\cap B^{s/2}_{p',2}(\R_t,L^{q'})},\\
\label{tracenormaleduham}
\bigg\|\Lambda^{-1}\bigg(\partial_y\int_0^te^{i(t-\tau)\Delta}f(\tau)d\tau\bigg)\bigg|_{y=0}
\bigg\|_{\mathcal{H}^s(\R_t)}&\lesssim &
\|f\|_{L^{p'}(\R_t,B^s_{q',2})\cap B^{s/2}_{p',2}(\R_t,L^{q'})}.
\end{eqnarray}
\end{prop}
\begin{proof}
We start with \eqref{stricduham} and \eqref{traceduham}.
As a first reduction, we point out that according to the usual Strichartz estimates 
(see \cite{Cazenave}, theorem 2.3.3 to corollary 2.3.9) and proposition \ref{estimcauchy}
\begin{eqnarray*}
\bigg\|e^{it\Delta}\int_{-\infty}^0e^{-is\Delta}f(s)ds|_{y=0}\bigg\|_{\mathcal{H}^s(\R_t)} 
\lesssim \bigg\|\int_{-\infty}^0e^{-is\Delta}f(s)ds\bigg\|_{H^s(\R^{d-1}\times \R^+)}
\lesssim \|f\|_{L^{p'}_tB^s_{q',2}},\\
\bigg\|e^{it\Delta}\int_{-\infty}^0e^{-is\Delta}f(s)ds\bigg\|_{L^p_tB^s_{q,2}} 
\lesssim \bigg\|\int_{-\infty}^0e^{-is\Delta}f(s)ds\bigg\|_{H^s}\lesssim \|f\|_{L^{p'}_tB^s_{q',2}},\\
\text{and }\bigg\|\partial_t\int_{-\infty}^0e^{i(t-s)\Delta}f(s)ds\bigg\|_{L^pL^{q}}
=\bigg\|\int_{-\infty}^0e^{i(t-s)\Delta}\Delta f(s)ds\bigg\|_{L^p_tL^{q}}
\lesssim \|f\|_{L^{p'}_tW^{2,q'}}.
\end{eqnarray*}
So, by interpolation
$$
\bigg\|e^{it\Delta}\int_{-\infty}^0e^{-is\Delta}f(s)ds\bigg\|_{B^{s/2}_{p,2}L^{q}}
\lesssim \|f\|_{L^{p'}B^s_{q',2}}.
$$
Therefore, it suffices to estimate $\int_{-\infty}^te^{i(t-s)\Delta}f(s)ds$, which is the 
solution of $i\partial_tu+\Delta u=if$, $\lim_{-\infty}u=0$. In this case, the analog of 
\eqref{stricduham} is also a consequence of the classical results in \cite{Cazenave}, 
and the analog of \eqref{traceduham} relies on the following duality argument. 
\paragraph{The case $s=0$}
We fix $g\in \mathcal{H}'(\R)$ and denote $v$ the solution of the backward  Neuman 
boundary value problem 
\begin{equation*}
\left\{
\begin{array}{ll}
i\partial_tv+\Delta v=0,\\
\lim_{+\infty}v(t)=0\\
\partial_yv|_{y=0}=g
\end{array}
\right.(x,y,t)\in\R^{d-1}\times \R^+\times \R.
\end{equation*}
According to the discussion p.\pageref{WPbackwards} and corollary \ref{BVPR}, this problem is well-posed and 
the solution is in $\cap_{(\rho,\gamma)\text{ admissible}}L^\rho_tL^\gamma$.
We extend $v$ on $\R^d\times \R_t$ by reflection
\begin{equation*}
 v(x,y,t)=
 \left\{
 \begin{array}{ll}
  v(x,y,t),\ y\geq 0,\\
  v(x,-y,t),\ y<0.
 \end{array}
\right.
\end{equation*}
In particular, $v|_{y=0^-}=v|_{y=0^+}$ and $\partial_yv|_{y=0^-}=-\partial_yv|_{y=0^+}=-g$. 
Using a density argument, the following integration by part is justified:
\begin{eqnarray*}
 \int_{\R_t}\int_{\R^{d}} if\overline{v}dxdydt=&=&\int_{\R_t}\int_{\R^d}u\overline{i\partial_t v+\Delta v}dxdydt \\
&&+\int_{\R_t}\int_{\R^{d-1}}-u|_{y=0}\overline{\partial_y v|_{y=0^-}}+u|_{y=0}\overline{\partial_yv|_{y=0^+}}dxdt \\
 &&+\int_{\R_t}\int_{\R ^{d-1}} \partial_yu|_{y=0}\overline{v|_{y=0^-}}-\partial_yu|_{y=0}\overline{v|_{y=0^+}}dxdt\\
 &=& 2\int_{\R_t}\int_{\R ^{d-1}}u|_{y=0}\overline{g}dxdt.
\end{eqnarray*}
Taking the sup over $\|g\|_{\mathcal{H}'}=1$, by duality  we deduce
\begin{equation}\label{forcings0}
\|u|_{y=0}\|_{\mathcal{H}(\R_t)}\leq\frac{1}{2}\|f\|_{L^{p'}(\R_t,L^{q'})}\sup_{\|g\|_{\mathcal{H}'}=1}
\|v\|_{L^{p}_tL^{q}}\lesssim \|f\|_{L^{p'}_tL^{q'}}. 
\end{equation}

\paragraph{Higher order estimates}
% Mettons qu'on cherche \`a faire des choses justes. On r\'esout le probl\`eme de Cauchy avec $u_0=0$, 
% $f\in L^{p'}W^{2,q'}$. Pour avoir les bonnes estim\'ees, il faut que $u|_{y=0}\in \mathcal{H}^2$, et en particulier 
% \begin{equation*}
% \partial_tu|_{y=0}\in \mathcal{H}.
% \end{equation*}
% Mais $\partial_tu|_{y=0}=(f-i\Delta u)|_{y=0}$. Comme $\Delta u$ est solution du probl\`eme de Cauchy avec $\Delta 
% u_0\in L^2$, $\Delta f\in L^{p'}L^{q'}$, on a d'apr\`es les r\'esultats en r\'egularit\'e nulle 
% $\Delta u|_{y=0}\in \mathcal{H}$. Mais les th\'eor\`emes de trace donnent seulement 
% $f|_{y=0}\in L^{p'}_tW^{2-1/q',q'}_x$.
% \\
% Mettons par exemple $f|_{y=0}\in L^1_tH^{3/2}_x$, alors (voir section sur $\mathcal{H}$), 
% \begin{equation*}
% \|f|_{y=0}\|_{\mathcal{H}}=\|e^{-it\Delta'}f|_{y=0}\|_{\dot{H}^{1/4}_tL^2_x}
% \end{equation*}
% Prenant $f|_{y=0}=e^{it\Delta'}f_1$, avec $f_1\in L^1_tH^{3/2}_x$, on a bien $f|_{y=0}\in L^1_tH^{3/2}_x$ mais aucune 
% raison pour que $\|f\|_{\mathcal{H}}=\|g\|_{\dot{H}^{1/4}_tL^2_x}<\infty$.\\
% Ca ne veut pas dire pour autant que la solution de \eqref{IBVP} ne satisfait pas les estimations escompt\'ees, mais 
% elle ne peut pas les satisfaire du fait d'estimations sur les op\'erateurs lin\'eaires.\\
% A rapprocher des r\'esultats de Holmer, selon lesquels les estimations $H^s$ avec for\c{c}age ne sont valables que pour 
% $s<1/2$ (comme par hasard indice critique d'existence de trace pour $f$...).
We recall that $\Delta'$ is the Laplacian in the $x$ variable.
If $f\in L^{p'}_tW^{2,q'}$, then $\Delta'u$ is the solution of
\begin{equation*}
i\partial_t\Delta'u+\Delta \Delta'u=\Delta'f,\ \lim_{-\infty} \Delta'u(t)=0,
\end{equation*}
therefore the estimate for $s=0$ implies $\|\Delta'u|_{y=0}\|_{\mathcal{H}}\lesssim 
\|\Delta'f\|_{L^{p'}_tL^{q'}}\lesssim \|f\|_{L^{p'}_tW^{2,q'}}$. By interpolation we get 
for $0<s<2$
\begin{equation}\label{regx}
 \int_{\R^d}\sqrt{||\xi|^2+\delta|}(1+|\xi|^{2s})|\widehat{u|_{y=0}}|^2d\delta d\xi\lesssim 
 \|f\|_{L^{p'}_tB^{s}_{q',2}}^2.
\end{equation}
Similarly, if $f\in W^{1,p'}L^{q'}$, then $\partial_tu$ satisfies
\begin{equation*}
i\partial_t\partial_tu+\Delta \partial_t u=\partial_tf,\ \lim_{-\infty}\partial_tu(t)=0,
\end{equation*}
the estimate for $s=0$ gives $\|\partial_tu|_{y=0}\|_{\mathcal{H}}\lesssim 
\|\partial_tf\|_{L^{p'}_tL^{q'}}$ and by interpolation again 
\begin{equation}\label{regt}
 \int_{\R^d}\sqrt{||\xi|^2+\delta||}(1+|\delta|^{s})|\widehat{u|_{y=0}}|^2d\delta d\xi\lesssim \|f\|_{B^s_{p',2}L^{q'}}^2.
\end{equation}
Combining \eqref{regx} and \eqref{regt} implies for $0<s<2$
\begin{equation*}
\|u|_{y=0}\|_{\mathcal{H}^s}\lesssim \|f\|_{B^{s/2}_{p',2}L^{q'}\cap L^{p'}_tB^{s}_{q',2}}.
\end{equation*}
\paragraph{Estimate \eqref{tracenormaleduham}} For $s=0$, we only sketch the similar duality argument : 
consider $v$ solution of the backward BVP with Dirichlet boundary condition $g$, 
and extend it on $\R^{d}\times \R_t$ as an odd function in the $y$ variable. The same computations as for 
\eqref{traceduham} lead to 
\begin{eqnarray*}
\sup_{g\in \mathcal{H}(\R_t)}\int_{\R_t\times \R^{d-1}} \partial_yu|_{y=0}\overline{g}dxdt
\lesssim \|f\|_{L^{p'}L^{q'}}\|g\|_{\mathcal{H}(\R_t)},\\
\Rightarrow \bigg\|\partial_yu|_{y=0}\bigg\|_{\mathcal{H}'(\R_t)}\lesssim 
\|f\|_{L^{p'}(L^{q'})},
\end{eqnarray*}
according to \eqref{traceNeuman}, this estimate is precisely \eqref{tracenormaleduham} for $s=0$. The case $0<s\leq 2$
follows from the same differentiation/interpolation argument.
\end{proof}

\begin{rmq}\label{rmqTimereg} 
The space $L^{p'}B^s_{q',2}\cap B^s_{p',2}L^{q'}$ seems natural at least scaling wise. In the case of dimension 
$1$, Holmer \cite{Holmer} managed to prove \eqref{traceduham} with only $\|f\|_{L^{p'}W^{s,q'}}$ in the right hand side
under the condition $s<1/2$. For $s\geq 1/2$, it is convenient to add some time regularity.
% Our guess is that for $s>1/2$,  some time smoothness is required, although the optimal 
% regularity is not clear.
\\
A (very formal) argument is as follows: suppose that $u$ is a smooth solution of $i\partial_tu+\Delta u=f,\ u|_{t=0}=0$. 
If $u|_{y=0}\in \mathcal{H}^2$, then $f|_{y=0}=i\partial_t(u|_{y=0}) +(\Delta u)|_{y=0}$, where 
$i\partial_tg\in \mathcal{H}$ and $w=\Delta u$ satisfies $i\partial_tw+\Delta w=\Delta f,\ w|_{t=0}=0$, so that 
the a priori estimate for $s=0$ gives $(\Delta u)|_{y=0}\in \mathcal{H}$. Therefore $f|_{t=0}$ should belong to 
$\mathcal{H}$, which can not be deduced from $f\in L^1_tH^2$. \\
Now if $f\in W^{1,1}_tL^2\cap L^1_tH^2$, from the numerology of Sobolev embeddings one expects 
\begin{eqnarray*}
f\in W^{3/4,1}_tH^{1/2}\Rightarrow \text{ ``almost'' } f|_{y=0}\in W^{3/4,1}_tL^2\hookrightarrow H^{1/4}_tL^2,\\
f\in W^{1/2,1}_tH^1\Rightarrow \text{ ``almost'' } f|_{y=0}\in W^{1/2,1}_tH^{1/2}\hookrightarrow L^2_tH^{1/2},  
\end{eqnarray*}
in particular, $f|_{y=0}\in H^{1/4}L^2\cap L^2H^{1/2}\hookrightarrow \mathcal{H}$. 
% This is also consistent with the 
% second order order compatibility condition of the IBVP \eqref{IBVP}
% \begin{equation*}
% i\partial_tg|_{t=0}+\Delta u_0|_{y=0}=f|_{y=t=0}.
% \end{equation*}
% Indeed if $g\in \mathcal{H}^3, \ u_0\in H^3$, one should have $f|_{y=t=0}\in H^{1/2}(\R^{d-1})$, which is consistent 
% with the \emph{formal} embedding $L^{1}_tH^3\cap W^{3/2,1}_tL^{2}\hookrightarrow W^{1,1}_tH^1\hookrightarrow C_{t}H^{1}
% \Rightarrow f|_{t=0}\in H^1(\R^{d-1}\times \R^+)$.
\end{rmq}

\subsection{Proof of theorems \ref{thtrivial} and \ref{maintheo} }\label{mainproof}
% \begin{prop}
% For $s\in ]0,2[\setminus\{1/2\}$, if the compatibility condition \textcolor{red}{ref \`a mettre, distinguer s=1/2} 
% is satisfied, 
% the full initial boundary value problem 
% \begin{equation}\label{IBVPDirichlet}
% \left\{
% \begin{array}{ll}
% i\partial_tu+\Delta u=f\in L^{p'}B^s_{q',2}\cap B^s_{p',2}L^{q'},\\
% u|_{t=0}=u_0\in H^s,\\
% u|_{\partial\Omega}=g\in \mathcal{H}^s(\R^+),
% \end{array}
% \right. (x,y,t)\in \R^{d-1}\times \R^+\times \R^+,
% \end{equation}
% has a unique solution $\displaystyle u\in \bigcap_{(p,q)\ \text{admissible}}
% L^p(\R^+,B^s_{q,2})\cap B^{s/2}_{p,2}(\R^+,L^q)$.
% \end{prop}
Up to using regularized data $u_0^n\in H^2_0,\ f_n\in W^{1,p'}_0L^{q'}\cap 
L^{p'}W^{2,q'}_0,\ g_n\in \mathcal{H}^2_0$ all quantities are well-defined, so we mainly focus on the issue of 
a priori estimates in this paragraph.
\paragraph{Proof of theorem \ref{thtrivial}}
First we point out a confusion to avoid for the operator $B$: if $B_{\R}$ is the Fourier multiplier 
with same symbol as $B$, $P_0$ the zero extension 
to $t\leq 0$, and $R$ the restriction to $t\geq 0$, we have 
\begin{equation*}
B=R\circ B_\R\circ P_0. 
\end{equation*}
We recall that $P_0$ (resp. $R$) is continuous $\mathcal{H}^s_0(\R^+)\to \mathcal{H}^s(\R)$, $s\neq 1/2$ (resp. 
$\mathcal{H}^s_0(\R)\to \mathcal{H}^s_0(\R^+)$), and by duality $P_0:\mathcal{H'}(\R^+)\to \mathcal{H}'(\R),\ R
:\ \mathcal{H}'(\R)\to \mathcal{H}'(\R^+)$ are continuous.
% This equation clarifies the fact that while $B_{\R}(a,b)\in \mathcal{H}^s(\R)$ if 
% $(a,\Lambda^{-1}b)\in (\mathcal{H}^s(\R))^2$, $B(a,b)\in \mathcal{H}^s(\R^+)$ only for 
% $(a,b)\in \mathcal{H}^s_0(\R^+)\times \Lambda\mathcal{H}^s_0(\R^+)$ ($\mathcal{H}^{1/2}_{00}$ if $s=1/2$). 
\subparagraph{The case $s=0$} We follow the method and notations from the beginning of section \ref{linestim}: 
let $v$ the solution of the Cauchy problem, $w$ the solution of \eqref{BVPaux2}, that is 
\begin{equation*}
\left\{
\begin{array}{ll}
i\partial_tv+\Delta v=f,\\
v|_{t=0}=u_0,
\end{array}\right.
\left\{
\begin{array}{ll}
i\partial_tw+\Delta w=0,\\
w|_{t=0}=0,\\
B(w|_{y=0},\partial_yw|_{y=0})=g-B(v|_{y=0},\partial_yv|_{y=0}),
\end{array}
\right. 
\end{equation*}
Since $\|v\|_{L^pL^q}\lesssim \|u_0\|_{L^2}+\|f\|_{L^{p'_1}_tL^{q'_1}}$ 
(Propositions \ref{estimcauchy} and \ref{estimDuham}), 
it suffices to check that $w$ exists and $\|w\|_{L^p_tL^q}\lesssim \|u_0\|_{L^2}+\|f\|_{L^{p'_1}_tL^{q'_1}}
+\|g\|_{\mathcal{H}}$.
Let us write $B_{\R}(a,b)=B_{1,\R}(a)+B_{2,\R}(b)$.
According to the Kreiss-Lopatinskii condition the 
symbols $b_1$ and $\sqrt{|\delta+|\xi|^2|}b_2$ are bounded uniformly in $(\delta,\xi)$.
From the estimates of section \ref{secCauchy}, $\|v|_{y=0}\|_{\mathcal{H}(\R_t)}+
\|\partial_yv|_{y=0}\|_{\mathcal{H}'(\R_t)}\lesssim \|u_0\|_{L^2}$,
this implies
\begin{eqnarray*}
\|B_{1,\R}\circ P_0\circ R (u|_{y=0})\|_{\mathcal{H}(\R_t)}&\lesssim& \|P_0\circ R u|_{y=0}\|_{\mathcal{H}(\R_t)}\lesssim 
\|u_0\|_{L^2}.\\
\|B_{2,\R}\circ P_0\circ R(\partial_yu|_{y=0})\|_{\mathcal{H}(\R_t)}^2
&=&\iint |b_2(\xi,\delta)|^2\sqrt{||\xi|^2+\delta|}\,
|\mathcal{F}_{x,t}\big(P_0\circ R(\partial_yu|_{y=0})\big)|^2d\xi d\delta \\
&\lesssim & \iint (||\xi|^2+\delta|)^{-1/2}\,
|\mathcal{F}_{x,t}\big(P_0\circ R(\partial_yu|_{y=0})\big)|^2d\xi d\delta \\
&=&\|P_0\circ R(\partial_yu|_{y=0})\|_{\mathcal{H'}(\R_t)}^2\lesssim \|u_0\|_{L^2},
\end{eqnarray*}
We can now apply proposition \ref{WPBVP} which gives the existence of $w$ with the expected Strichartz estimate. \\
The causality follows by taking the difference of two solutions and using the property on support of solutions 
in Corollary \ref{BVPR}.
\subparagraph{The case $s=2$} Here we assume $f\in L^{p'}_tW^{2,q'}\cap W^{1,p'}_tL^{q'}$, 
$u_0\in H^2_0(\R^{d-1}\times \R^+),\ g\in \mathcal{H}^2_0(\R^+)$. 
According to proposition \ref{WPBVP}, we can use again a superposition principle provided 
$$B(v|_{y=0},\partial_yv|_{y=0})\in \mathcal{H}^2_0(\R^+)\text{ or equivalently }
B_{\R}\circ P_0\circ R(v|_{y=0},\partial_yv|_{y=0})\in \mathcal{H}^2(\R),$$ 
since $B_{\R}\circ P_0$ is supported in $t\geq 0$.
By assumption, 
$v|_{y=t=0}=u_0|_{y=0}=0$, therefore estimate \eqref{traceCauchy} and corollary \ref{normeloc} imply
\begin{equation*}
\|B_{1,\R}\circ P_0\circ R(v|_{y=0})\|_{\mathcal{H}^2(\R)}\lesssim 
\|u_0\|_{H^2}+\|f\|_{L^{p'}W^{2,q'}\cap W^{1,p'}L^{q'}}.
\end{equation*}
Moreover, estimate \eqref{traceCauchy} also implies
\begin{eqnarray*}
\|\Delta'\partial_yv|_{y=0}\|_{\mathcal{H}'(\R_t)}^2+\|\partial_t\partial_yv|_{y=0})\|_{\mathcal{H}'(\R_t)}^2
&\sim &
\iint_{\R^d} \frac{|\widehat{\partial_yv|_{y=0}}|^2}{\sqrt{||\xi|^2+\delta|}}(|\xi|^2+|\delta|)^2d\xi d\delta\\
&\lesssim &\|u_0\|_{H^2}^26+\|f\|_{L^{p'}W^{2,q'}\cap W^{1,p'}L^{q'}}^2.
\end{eqnarray*}
But since $u_0\in H^2_0(\R^{d-1}\times \R^+)$, 
$\partial_yv|_{y=t=0}=\partial_yu_0|_{y=0}=0$ thus 
\begin{equation*}
\partial_tP_0\circ R(\partial_yv|_{y=0})=P_0\circ R(\partial_t\partial_yv|_{y=0}),\ 
\Delta'P_0\circ R(\partial_yv|_{y=0})=P_0\circ R(\Delta'\partial_yv|_{y=0}).
\end{equation*}
By continuity of $P_0\circ R:\ \mathcal{H}'\to\mathcal{H}'$, 
$(P_0\circ R(\partial_t\partial_yv|_{y=0}),\ P_0\circ R(\Delta'\partial_yv|_{y=0}))\in(\mathcal{H'})^2$. Finally, 
using the boundedness of $b_2\sqrt{||\xi|^2+\delta|}$ we get 
\begin{eqnarray*}
 \|B_{2,\R}\circ P_0\circ R(\partial_yv|_{y=0})\|_{\mathcal{H}^2}&\lesssim& \|\partial_tP_0\circ R(\partial_yv|_{y=0})
 \|_{\mathcal{H'}}
 +\|\Delta' P_0\circ R(\partial_yv|_{y=0})\|_{\mathcal{H'}}\\
& \lesssim &\|u_0\|_{H^2}+\|f\|_{L^{p'}W^{2,q'}\cap W^{1,p'}L^{q'}}
\end{eqnarray*}
which implies as expected $B_{2,\R}\circ P_0\circ R(\partial_yv|_{y=0})\in \mathcal{H}^2(\R)$.
\subparagraph{The case $0<s<2$} After fixing an extension operator, since $(u_0,f)\to B(v|_{y=0},\partial_yv|_{y=0})$ is 
continuous $L^2\times L^{p'}_tL^{q'}\to \mathcal{H}(\R^+)$ and $H^2_0\times \big(W^{1,p'}_tL^{q'}\cap L^{p'}_tW^{2,q'})
\to \mathcal{H}^2_0(\R^+)$, the general case follows by interpolation.
\paragraph{Proof of theorem \ref{maintheo}} Let $s\in [0,2]$. 
We fix extensions of  $u_0,f$ to $y\leq 0$ and solve
\begin{equation*}
\left\{
\begin{array}{ll}
i\partial_tv+\Delta v=f,\\
v|_{t=0}=u_0,
\end{array}
\right. (x,y,t)\in \R^d\times \R.
\end{equation*}
From the estimates for the Cauchy problem, $v|_{y=0}\in \mathcal{H}^s(\R)$. Consider the BVP
\begin{equation}\label{BVPaux}
\left\{
\begin{array}{ll}
i\partial_tw+\Delta w=0,\\
w|_{t=0}=0,\\
w|_{y=0}=g-v|_{y=0},
\end{array}
\right. (x,y,t)\in \R^{d-1}\times \R^+\times \R^+.
\end{equation}
If $s>1/2$, the trace $v|_{y=t=0}=u_0|_{y=0}$ is well defined and belong to $H^{s-1/2}$. Moreover
the compatibility condition imposes $(g-v|_{y=0})|_{t=0}=g|_{t=0}-u_0|_{y=0}=0$ so that for $s\in [0,2]\setminus\{1/2\}$,
$g-v|_{y=0}\in \mathcal{H}^s_0(\R^+)$. From proposition \ref{WPBVP} there exists a unique solution $w\in C(\R_t^+,H^s)$
to \eqref{BVPaux}. Now $u:=v|_{y\geq 0}+w$ is a solution of \eqref{IBVP}, it satisfies the expected 
estimate because according to propositions \ref{WPBVP}, \ref{estimcauchy} and \ref{estimDuham},
$v$ and $w$ do.\\
In the case $s=1/2$, we first note that 
$$\forall\,t\geq 0,\ \int_0^te^{i(t-s)}\Delta f(s)ds=\int_0^te^{i(t-s)\Delta} P_0\circ R(f(s))ds,$$ 
and $P_0\circ R(f) \in B^{1/4}_{p,2}(\R_t,L^q)\cap L^p(\R_t,B^{1/2}_{q,2})$. From Proposition \ref{estimDuham}
$\int_0^te^{i(t-s)\Delta }P_0\circ Rf(s)ds|_{y=0}\in \mathcal{H}^{1/2}(\R)$, and vanishes for $t\leq 0$, therefore 
$R\big(\int_0^te^{i(t-s)\Delta} f(s)ds|_{y=0})\in \mathcal{H}^{1/2}_{00}(\R^+)$ (see definition \eqref{def00}). 
In order to solve \eqref{BVPaux}, we are left to prove that 
if the compatibility condition is satisfied, then $(e^{it\Delta}u_0)|_{y=0}-g\in \mathcal{H}^{1/2}_{00}(\R^+)$. 
From the previous estimates, we know $(e^{it\Delta}u_0)|_{y=0}-g\in \mathcal{H}^{1/2}(\R^+)$, and we must check 
condition \eqref{condition00}, that is :
\begin{equation*}
\iint_{\R^+\times \R^{d-1}}\frac{\big|(e^{it\partial_y^2}u_0)|_{y=0}-e^{-it\Delta'}g(x,t)\big|^2}{t}dtdx<\infty.
\end{equation*}
Using the change of variable $t\to \sqrt{t}$, the compatibility condition \eqref{compaglob} ensures 
\begin{equation*}
\iint_{\R^+\times \R^{d-1}}\frac{\big|u_0(x,\sqrt{t})-e^{-it\Delta'}g(x,t)\big|^2}{t}dtdx<\infty.
\end{equation*}
Therefore we only need to estimate $u_0(x,\sqrt{t})-(e^{it\partial_y^2}u_0)|_{y=0}$. We use the following 
interpolation argument: if $u_0\in H^1(\R^d)$, the identity $u_0(x,\sqrt{t})-(e^{it\partial_y^2}u_0)|_{y=0}=
u_0(x,\sqrt{t})-u_0(x,0)+u_0(x,0)-(e^{it\partial_y^2}u_0)|_{y=0}$ makes sense, and thanks to Hardy's inequality 
\begin{equation*}
\iint_{\R^+\times \R^{d-1}}\frac{ | u_0(x,\sqrt{t})-u_0(x,0) |^2}{t^{3/2}}dtdx= 
2\iint_{\R^+\times \R^{d-1}}\frac{ | u_0(x,y)-u_0(x,0) |^2}{y^{2}}dydx\lesssim \|\partial_yu_0\|_{L^2}^2.
\end{equation*}
Similarly, the sharp Kato smoothing \eqref{traceCauchy} implies $\|(e^{it\partial_y^2}u_0)|_{y=0}\|_{\dot{H}^{3/4}_tL^2}
\lesssim \|u_0\|_{H^1}$ so that the (fractional) Hardy's inequality gives 
\begin{equation*}
\iint_{\R^+\times \R^{d-1}}\frac{ | (e^{it\partial_y^2}u_0)|_{y=0}-u_0(x,0)|^2}{t^{3/2}}dtdx\lesssim 
\|(e^{it\partial_y^2}u_0)|_{y=0}\|_{\dot{H}^{3/4}_tL^2}^2\lesssim \|u_0\|_{H^1}^2.
\end{equation*}
On the other hand, we have by a similar simpler argument
\begin{equation*}
\iint_{\R^+\times \R^{d-1}}\frac{|u_0(x,\sqrt{t})|^2+|(e^{it\partial_y^2}u_0)|_{y=0}|^2}{t^{1/2}}\lesssim 
\|u_0\|_{L^2}^2+\|(e^{it\partial_y^2}u_0)|_{y=0}\|_{H^{1/4}}^2\lesssim \|u_0\|_{L^2}^2.
\end{equation*}
We deduce by interpolation 
\begin{equation*}
\iint_{\R^+\times \R^{d-1}}\frac{|u_0(x,\sqrt{t})-(e^{it\partial_y^2}u_0)|_{y=0}|^2}{t}dt\lesssim \|u_0\|_{H^{1/2}}^2.
\end{equation*}
This implies $(e^{it\Delta}u_0)|_{y=0}-g\in \mathcal{H}^{1/2}_{00}$. Clearly, the argument is independent of 
the choice of the extension operator, and we can end the proof as for the case $s\neq1/2$.

\section{Local and global existence}\label{WPnonlin}
For simplicity, we only consider nonlinearities of the type $\varepsilon|u|^{a-1}u$, $a>1,\ \varepsilon\in \{-1,1\}$
Dirichlet boundary conditions, $u_0\in H^1$.
More general nonlinearities and indices of regularity can be treated with similar methods, see  
chapter $4$ from \cite{Cazenave}. \\
Since so far we have always considered global solution, some clarifications for local solutions of nonlinear problems 
are required.
For $P_T$ an extension operator as in lemma \ref{extOp}, consider the map $\Phi:v\in L^\infty(\R_t,H^1)\mapsto \Phi(v)$ 
the solution of 
\begin{equation}\label{OpPhi}
\left\{
\begin{array}{ll}
i\partial_tu+\Delta u=\varepsilon |P_Tv|^{a-1}P_Tv,\\
u|_{t=0}=u_0,\\
u|_{y=0}=g.
\end{array}
\right.
\end{equation}
If $1<a<1+4/(d-2)$, $2<a+1<2d/(d-2)$ thus by Sobolev's embedding $v\in L^\infty_tL^{a+1}$. If $(p,a+1)$ is admissible, 
we deduce $|P_Tv|^{a-1}P_Tu\in L^{p'}_t\cap L^{(a+1)'}$, and according to theorem \ref{maintheo} 
$\Phi$ is well-defined $L^\infty_tH^1\to C_tL^2$.\\
% $H^1\hookrightarrow L^{2d/(d-2)}$ ($L^p$, $p<\infty$ if $d=2$), 
We say that $u$ is a local solution on $[0,T]$ of 
\begin{equation}\label{NLS}
\left\{
\begin{array}{ll}
i\partial_tu+\Delta u=\varepsilon |u|^{a-1}u,\\
u|_{t=0}=u_0,\\
u|_{y=0}=g.
\end{array}
\right.
\end{equation}
if $u$ is the restriction on $[0,T]$ of a fixed point of $\Phi$.
\begin{theo}
Let $(u_0,g)\in H^1(\R^{d-1}\times \R^+)\times \mathcal{H}^1(\R_t^+)$ such that 
$u_0|_{y=0}=g|_{t=0}$, $1<a<1+4/(d-2)$. The IBVP \eqref{NLS}
has a unique maximal solution in $C([0,T_{\max}),H^1)$. If $T_{\max}<\infty$, $\displaystyle 
\overline{\lim_{T_{\max}}} \|u(t)\|_{H^1}=\infty$. For any $T$ such that $u$ exists on $[0,T]$ and 
$(p,q)$ an admissible pair, then $u\in L^p([0,T],W^{1,q})\cap B^{1/2}_{p,2}L^q$. \\
If moreover $1+4/d\leq a$, there exists $\varepsilon>0$ such that if $\|u_0\|_{H^1}+\|g\|_{\mathcal{H}^1}<\varepsilon$
then the solution is global.
\end{theo}

\begin{proof}
We use the convenient notation $L^{1/p}=L^p$.
Let us recall shortly the classical Kato's argument, with some modifications to handle time 
regularity.
\paragraph{Local existence}
% We first note that $u$ solves \eqref{NLS} on $[0,T]$ if and only if 
% for $\lambda>0$, 
% $u_\lambda:=\lambda^{2/(a-1)}u(\lambda x,\lambda y,\lambda^2t)$ solves 
% \begin{equation}
% \left\{
% \begin{array}{ll}
% i\partial_tu_{\lambda}+\Delta u_{\lambda}=\varepsilon |u_{\lambda}|^{a-1}u_{\lambda},\\
% u_{\lambda}|_{t=0}=\lambda^{2/(a-1)}u_0(\lambda x,\lambda y),\\
% u_{\lambda}|_{y=0}=\lambda^{2/(a-1)}g(\lambda x,\lambda^2t).
% \end{array}
% \right.
% \end{equation}
% on $[0,T/\lambda^2]$. Moreover, 
% \begin{equation*}
% \|\lambda^{2/(a-1)}u_0(\lambda x,\lambda y)\|_{H^1}+\|\lambda^{2/(a-1)}g(\lambda x,\lambda^2t)\|_{\mathcal{H}^1(\R^+)} 
% \lesssim \lambda^{2/(a-1)+1-d/2}\big(\|u_0\|_{H^1}+\|g\|_{\mathcal{H}^1(\R^+)}\big),
% \end{equation*}
% with by assumption $2/(a-1)+1-d/2>0$. Therefore solving \eqref{NLS} for $t\in [0,\lambda^2]$, $\lambda<<1$ is equivalent 
% to solve \eqref{NLS} for $t\in [0,1]$ with small data. \\
% Let us now assume $\|g\|_{\mathcal{H}^1}+\|u_0\|_{H^1}\leq \varepsilon$, set $q=a+1<2d/(d-2)$, and $p$ such that 
% $2/p+d/q=d/2$. 
For $M$ to fix later, we set $S$ the ball of radius $M$ in
$L^\infty(\R^+,H^1)\cap L^p(\R^+,W^{1,q})\cap B^{1/2}_{p,2}(\R^+,L^q)$, $q=a+1$, $(p,q)$ admissible.
We use on $S$ the following distance
\begin{equation*}
d(u,v)=\|u-v\|_{L^\infty(\R^+,L^2)\cap L^q(\R^+,L^r)}.
\end{equation*}
$(S,d)$ is a complete set (see e.g. \cite{Cazenave} section 4.4). 
We fix an extension operator $P_T$ as in lemma \ref{extOp}:
% \begin{equation*}
% P_T:\ L^\infty([0,T], H^1)\cap L^p([0,T],B^{1/2}_{q,2})\to 
% L^\infty(\R, H^1)\cap L^p(\R,B^{1/2}_{q,2}),
% \end{equation*}
such that for any  $v\in S$, 
\begin{equation}\label{timesupp}
\text{supp}(P_Tv)\subset \R^{d-1}\times \R^+\times [-T,2T],\ \|P_Tv\|_{B^{1/2}_{p,2}(\R_t,L^q)}\lesssim 
T^{1/p-1/2}\|v\|_{B^{1/2}_{p,2}(\R_t,L^q)},
\end{equation}
% $B^{1/2}_{p',2}([0,1],L^{q'})\to B^{1/2}_{p',2}(\R,L^{q'})$, 
and we construct a fixed point to $\Phi$, with $\Phi$ defined at \eqref{OpPhi}.\\
Combining the inclusions $B^1_{q,2}\subset W^{1,q}$, $B^1_{q',2}\supset W^{1,q'}$ (see \cite{berglof} 
theorem 6.4.4), with the linear estimates of theorem \ref{maintheo} we get
\begin{equation*}
\|\Phi(v)\|_{L^\infty_tH^1\cap L^p_tW^{1,q}\cap B^{1/2}_{p,2}(\R^+,L^q)}\lesssim \|u_0\|_{H^1}+\|g\|_{\mathcal{H}^1}
+\||P_Tv|^{a-1}P_Tv\|_{L^{p'}W^{1,q'}\cap B^{1/2}_{p',2}L^{q'}}.
\end{equation*}
Using $aq'=q$, $\frac{1-2/q}{a-1}=1/q$, the embedding $H^1\hookrightarrow L^q$ and assumption 
\eqref{timesupp}, we have
\begin{eqnarray}\nonumber
\||P_Tv|^{a-1}P_Tv\|_{L^{p'}(\R,W^{1,q'})}&\lesssim& \|P_Tv\|^a_{L^{ap'}_tL^{q}}
+\|P_Tv\|_{L^{\infty}_tL^q}^{a-1}\|\nabla P_Tv\|_{L^p_tL^q}
\|1\|_{L^{1-2/p}([-T,2T],L^\infty)}
\\
&\lesssim& T^{1-1/(ap')}\|v\|^a_{L^{\infty}_tH^1}
+T^{1-2/p}\|P_Tv\|_{L^{\infty}H^1}^{a-1}\|\nabla P_Tv\|_{L^p_TL^q}
\\
\label{estimduham}&\lesssim& (T^{1-1/(ap')}+T^{1-2/p})M^a.
\end{eqnarray}
Similarly for the time regularity, we have using proposition \ref{compoFrac} and lemma \ref{extOp}
\begin{eqnarray*}
\||P_Tv|^{a-1}P_Tv\|_{B^{1/2}_{p',2}(\R,L^{q'})}&\lesssim& \|(P_Tv)^{a-1}\|_{L^{1-2/p}_TL^{1-2/q}}\|
P_Tv\|_{B^{1/2}_{p,2}L^q}\\
&\lesssim& T^{1-2/p}\|v\|_{L^{\infty}_tL^{q}}^{a-1}T^{1/p-1/2}\|v\|_{B^{1/2}_{p,2}L^q}\\
&\lesssim& T^{1/2-1/p}M^a.
\end{eqnarray*}
Therefore for $0\leq T\leq 1$,
\begin{equation*}
\|\Phi(v)\|_{L^\infty H^1\cap L^pW^{1,q}\cap B^{1/2}_{p,2}L^q}\lesssim \|u_0\|_{H^1}+\|g\|_{\mathcal{H}^1}
+ T^{1/2-1/p}M^a.
\end{equation*}
Choosing $M>\|u_0\|_{H^1}+\|g\|_{\mathcal{H}^1}$, $T$ small enough, $\Phi$ maps $S$ into $S$. 
Then from similar computations 
% \footnote{On pourrait aussi se contenter de la contractivit\'e 
% pour la m\'etrique plus faible mais quand m\^eme compl\`ete $L^\infty L^2\cap L^pL^q$, qui demande moins de 
% r\'egularit\'e sur $f$. Bon...}
% \begin{equation*}
% \|\Phi(u_1)-\Phi(u_2)\|_{S}\lesssim \big(T^{1-1/(ap')}+T^{1-\frac{2}{p}}\big)
% (\|u_1\|_S+\|u_2\|_S)^{a}\|u_1-u_2\|_S.
% \end{equation*}
% For the uniqueness in $C_tH^1$, we observe that if $v$ is another solution the same computations give 
\begin{equation}\label{estimuniq}
\|\Phi(u)-\Phi(v)\|_{L^\infty_t L^2\cap L^p_tL^q}\lesssim 
% T^{1-2/p}\|u-v\|_{L^p_tL^q}(\|u\|_{L^\infty_t H^1}
% +\|v\|_{L^\infty_t H^1})^{a-1}\leq T^{1-2/p}
T^{1-2/p}(\|u\|_{L^\infty_tH^1}+\|v\|_{L^\infty_t H^1})^{a-1}\|u-v\|_{L^p_t L^q}.
\end{equation}
Up to decreasing $T$, the usual fixed point argument gives the existence of a unique fixed point in $S$ 
for $T$ small enough. Estimate \eqref{estimuniq} also implies uniqueness in $L^\infty H^1$, and by causality 
the solution does not depend on the choice of the extension operator.\\
Thanks to the local well-posedness in $H^1$, the existence and uniqueness of a maximal solution follows.
\paragraph{Global existence} Let us go back to \eqref{estimduham}, assuming $a\geq 1+4/d$. 
Then $\displaystyle \frac{1}{p}=\frac{d(a-1)}{4(a+1)}$ and
\begin{eqnarray*}
\frac{1}{a-1}\bigg(1-\frac{2}{p}\bigg)-\frac{1}{p}=\frac{1}{a-1}\bigg(1-\frac{a+1}{p}\bigg) 
=\frac{1}{a-1}\bigg(1-\frac{d(a-1)}{4}\bigg)\leq 0,\\
\frac{1}{a}\bigg(1-\frac{1}{p}\bigg)-\frac{1}{p}=\frac{1}{a}\bigg(1-\frac{a+1}{p}\bigg)
=\frac{1}{a}\bigg( 1-\frac{d(a-1)}{4}\bigg)\leq 0.
\end{eqnarray*}
Therefore $L^{ap'}\cap L^{\frac{1}{a-1}\big(1-\frac{2}{p}\big)}\subset L^\infty\cap L^p$. 
As we work with small data, we can assume that the solution exists on $[0,T_0]$, $T_0\geq 1$,
and for any $T\geq T_0$, using 
$H^1\hookrightarrow L^q$
\begin{eqnarray*}
\||u|^{a-1}u\|_{L^{p'}_TW^{1,q'}}&\lesssim& \|u\|^a_{L^{ap'}_TL^{q}}
+\|u\|_{L^{\frac{1}{a-1}\big(1-\frac{2}{p}\big)}_TL^q}^{a-1}\|\nabla u\|_{L^p_TL^q}
\\
&\lesssim& \|u\|_{L^\infty_TH^1\cap L^pW^{1,q}}^a.
\end{eqnarray*}
The same computations can be applied to estimate time regularity, so that setting 
\\ 
$m(T)=\|u\|_{L^\infty_TH^1\cap L^p_TW^{1,q}\cap B^{1/2}_{p,2}([0,T],L^q)}$, we have with $C$ independent of 
$T\geq T_0$
\begin{equation*}
m(T)\leq C(\|u_0\|_{H^1}+\|g\|_{\mathcal{H}^1}+m(T)^a).
\end{equation*}
If $\|u_0\|_{H^1}+\|g\|_{\mathcal{H}^1}\leq \varepsilon$ small enough, then from the fixed point argument 
$m(1)\leq A\varepsilon$ for some $A>0$. Choosing $B> \max(A,C)$ and $\varepsilon$ small enough 
such that $C+CB^a\varepsilon^{a-1}<B$, for any 
$T\in [0,T_{\max}[$, $m(T)\leq B\varepsilon$ thus $T_{\max}=\infty$.
\end{proof}
% \begin{rmq}
% It should be possible to prove local well-posedness without a rescaling argument, but the natural method would 
% be to use the extension operator $P_T: u\mapsto (P_1u(T\cdot))(t/T)$ which maps $B^{1/2}_{p,2}([0,T])$ to functions 
% supported in $[-T,2T]$. The norm of the operator is unbounded as $T\to 0$ if $2p>1$, therefore it creates in the 
% a priori estimates a competition between the $T^\varepsilon$ gained in H\"older-like inequalities and the 
% operator norm, leading to a more intricate choice of integrability indices.
% % It does not seem straightforward to prove local well-posedness without a rescaling argument, indeed a direct approach 
% % would require to use an extension operator valued in functions supported in $|t|\lesssim T$, which is likely unbounded 
% % if $\|u\|_{L^\infty H^1}$ is not small, there is no reason for the restriction norm $\|u\|_{B^{1/2}_{p,2}L^q}$ to 
% % be small either, due to the embedding $B^{1/2}_{p,2}\to C^0$ if $p>2$. 
% \end{rmq}

\begin{rmq}
For the Schr\"odinger equation on $\R^d$, global well-posedness for small data is known provided 
$a_S<a$, where $a_S=(\sqrt{d^2+12d+4}+d+2)/(2d)<1+4/d$ is the so-called Strauss exponent, see \cite{Strauss}. 
Strichartz estimates for ``non admissible pairs'' (\cite{Cazenave}, section 2.4) are the missing tool for 
reaching this range.
\end{rmq}

\section{Asymptotic behaviour}\label{scattering}
The aim of this section is to show that the global small solution constructed in section  \ref{WPnonlin} scatters 
in the sense that it is asymptotically linear. For the Cauchy problem, the classical definition\footnote{up to some 
flexibility for the functional settings.} is 
\begin{equation*}
\exists\,\varphi\in H^1:\ \lim_{t\rightarrow +\infty}\|e^{-it\Delta} u(t)-\varphi\|_{H^1}=0.
\end{equation*}
We propose a natural extension for the Dirichlet boundary value problem: we define the operator 
$\Phi_g(t,u_0)=v(t,\cdot)$ where $v$ is the solution of
\begin{equation*}
 \left\{
\begin{array}{ll}
i\partial_tv+\Delta v=0,\\
v|_{t=0}=u_0,\\
v|_{x_d=0}=g.
 \end{array}
\right.
\end{equation*}
Note that if $g$ is defined for $t\in \R$, by reversibility of the boundary value problem with Dirichlet boundary conditions, 
$\Phi_g(t,u_0)$ is well defined for $t\in \R$.
By uniqueness of the solution, $\Phi_{g(t+\cdot)}(-t,\cdot)\circ \Phi_g(t,u_0)=u_0$, thus 
$u\to \Psi_g(t,u):=\Phi_{g(t+\cdot)}(-t,u)$ is the inverse of $\Phi_g(t,\cdot)$. This leads to the natural definition 
for scattering:
\begin{define}
 If $u$ is a global solution to \eqref{NLS}, we say that it scatters in $H^1$ if 
 \begin{equation*}
\lim_{t\rightarrow \infty}\Psi_g(t,u(t)) \text{ exists in }H^1.  
 \end{equation*}
\end{define}
\begin{rmq}
 Of course, it is equivalent to the more ``forward'' definition 
\begin{equation*}
\exists\,\varphi\in H^1:\ \lim_t \|\Phi_g(t,\varphi)-u(t)\|_{H^1}=0,
\end{equation*}
which has the advantage of making sense for non reversible BVP (but is not as easily checked).
\end{rmq}

\begin{prop}
The global solution constructed in section \ref{WPnonlin} scatters in $H^1$.
\end{prop}
\begin{proof}
It suffices to check that $\Psi_g(t,u(t))$ is a Cauchy sequence. We keep the same notation as in the previous 
section. For $t>s$,  we have
$\Psi_g(t,u(t))-\Psi_g(s,u(s))=w(0)$ where $w$ is the solution of 
\begin{equation*}
 \left\{
\begin{array}{ll}
i\partial_rw+\Delta w=0,\\
w|_{r=t}=u(t)-\Phi_{g(s+\cdot)}(t-s,u(s)),\\
w|_{y=0}=0.
 \end{array}
\right.
\end{equation*}
On the other hand, $u(t)-\Phi_{g(s+\cdot)}(t-s,u(s))$ is the value at time $t$ of the solution of
\begin{equation*}
 \left\{
\begin{array}{ll}
i\partial_rz+\Delta z=|u|^{a-1}u1_{r\geq s},\\
z|_{r=s}=u(s)-u(s)=0,\\
z|_{y=0}=0.
 \end{array}
\right.
\end{equation*}
We deduce $\|\Psi_g(t,u(t))-\Psi_g(s,u(s))\|_{H^1}\lesssim \||u|^{a-1}u\|_{L^{p'}([s,\infty[,W^{1,q'})\cap 
B^{1/2}_{p',2}([s,\infty[,L^{q'})}\rightarrow_s 0$, therefore by Cauchy's criterion $\Psi_g(t,u(t))$ converges in $H^1$.
\end{proof}
Due to the presence of boundary conditions, there is some ``room'' for other definitions of scattering. The purpose 
of the next proposition is to show that the asymptotic behaviour is actually trivial, in the sense that the 
solution converges to the restriction on $y\geq 0$ of $e^{it\Delta}\varphi$ for some $\varphi\in H^1(\R^d)$.
We denote $\Delta_D$ the Dirichlet laplacian.

\begin{prop}
There exists  $\varphi\in H^1_0$ such that $\|u(t)-e^{it\Delta_D}\varphi\|_{H^1}
\rightarrow_t0$. Equivalently, $u$ converges as $t\to \infty$ to the restriction on $y\geq 0$ of the  
solution of 
\begin{equation*}
\left\{
\begin{array}{ll}
 i\partial_tv+\Delta v=0,\\
 v|_{t=0}=A(\varphi),
\end{array}
\right. x\in \R^d
\end{equation*}
where $A(\varphi)$ is the antisymetric extension on $y\leq 0$ of $\varphi$.
\end{prop}
\begin{proof}
Let us fix $\mathcal{R}$ a lifting operator $H^{1/2}(\R^{d-1})\rightarrow H^1(\R^{d-1}\times 
\R^+)$, $\mathcal{P}$ an extension operator $\mathcal{H}^1(\R^+_t)\rightarrow 
\mathcal{H}^1(\R_t)$. We define 
$\mathcal{P}_t: \mathcal{H}^1([t,\infty[\times \R^{d-1})\rightarrow 
\mathcal{H}^1(\R\times \R^d),\ g\rightarrow \mathcal{P}(g(\cdot+t))(\cdot-t)$.
We define now a modified backward operator :
\begin{equation*}
\widetilde{\Psi}_g(t,u(t)):=\Phi_{\mathcal{P}_t(g)}(-t,u(t)). 
\end{equation*}
For $t>s$, $\widetilde{\Psi_g}(t,u(t))-\widetilde{\Psi_g}(s,u(s))=w(0)$ where $w$ is the solution of 
\begin{equation*}
 \left\{
\begin{array}{ll}
i\partial_\tau w+\Delta w=0,\\
w|_{\tau=t}=u(t)-\Phi_{g(s+\cdot)}(t-s,u(s)),\\
w|_{y=0}=\mathcal{P}_tg-\mathcal{P}_sg.
 \end{array}
\right.
\end{equation*}
We already know (see the previous proof) that $$\|u(t)-\Phi_{g(s+\cdot)}(t-s,u(s))\|_{H^1}\rightarrow_s 0,$$
moreover $\displaystyle \lim_{t\to \infty}\|g\|_{\mathcal{H}^1([t,\infty[)}=0$ (corollary \ref{normeloc} point 3), 
thus $\|\mathcal{P}_t(g)\|_{\mathcal{H}^1}\rightarrow_\infty  0$. We deduce 
$\|\widetilde{\Psi_g}(t,u(t))-\widetilde{\Psi_g}(s,u(s))\|_{H^1}\rightarrow_{s,t}0$, thus from Cauchy's criterion 
$\widetilde{\Psi_g}(t,u(t))$ converges in $H^1$, we set $\lim_t\widetilde{\Psi_g}(t,u(t))=\varphi$.
We remark now that  $\Phi_{\mathcal{P}_tg}(\tau,u(t))-e^{i\tau\Delta_D}(u(t)-\mathcal{R}g(t))$ 
is the solution of 
\begin{equation*}
 \left\{
\begin{array}{ll}
i\partial_\tau w+\Delta w=0,\\
w|_{\tau=0}=\mathcal{R}g(t),\\
w|_{y=0}=\mathcal{P}_tg.
 \end{array}
\right.
\end{equation*}
Since $\|g\|_{\mathcal{H}^1([t,\infty[)}\to_\infty 0$, and from the embedding $\mathcal{H}^1\hookrightarrow 
C([t,\infty|,H^{1/2})$, we have $\mathcal{R}g(t)\rightarrow_t0$, this implies 
$\displaystyle \lim_{t\to \infty}\|\Phi_{\mathcal{P}_tg}(\tau,u(t))-e^{i\tau\Delta_D}(u(t)
-\mathcal{R}g)\|_{L^\infty_\tau H^1}\rightarrow 0$. In particular for $\tau=-t$
\begin{equation*}
\lim_{t\to \infty} \|\widetilde{\Psi}_{g}(t,u(t))-e^{-it\Delta_D}(u(t)-\mathcal{R}g(t))\|_{H^1}=0. 
\end{equation*}
As $\widetilde{\Psi}_g(t,u(t))\rightarrow_t \varphi\in H^1$, we deduce
$e^{-it\Delta_D}(u(t)-\mathcal{R}g(t)))\rightarrow_t\varphi$ too.  Furthermore for any $t$, 
$u(t)-\mathcal{R}g(t)\in H^1_0$ which is closed so $\varphi \in H^1_0$, and $e^{it\Delta_D}\varphi\in H^1_0$.
Finally from  $\|\mathcal{R}g(t)\|_{H^1}\rightarrow_t0$ we conclude 
$\|u(t)-e^{it\Delta_D}\varphi\|_{H^1}\rightarrow 0$. 
\end{proof}

\appendix 
\section{Remarks on the optimality of \texorpdfstring{$\mathcal{H}$}{H}}
A natural question is wether $\mathcal{H}$ is the weakest space for which the solution to \eqref{IBVP} is $C_tL^2$.
We consider the BVP
\begin{equation*}
\left\{
\begin{array}{ll}
 i\partial_tu+\Delta u=0,\\
\displaystyle \lim_{-\infty}u(t)=0,\\
 u|_{y=0}=g.
\end{array}
\right.
\end{equation*}
We formulate our problem as follows
\begin{equation}\label{questopt}
\text{Is there a weight }p>0\text{ such that }\|u\|_{C_tL^2}\lesssim 
\big(\int |\widehat{g}|^2p(\xi,\delta)d\delta d\xi\big)^{1/2}
\text{ and }\inf \dfrac{p(\xi,\delta)}{\sqrt{||\xi|^2+\delta|}}=0\ ?
\end{equation}
The aim of this section is to show that the answer to this question is positive, 
even under the stronger assumptions that $p\leq \sqrt{||\xi|^2+\delta|}$ and for any $\lambda>0$, 
$p(\lambda \xi,\lambda^2\delta)=\lambda p(\xi,\delta)$.
However we will see that region where the inf is realized is a bit peculiar.

We recall that the solution is given by $\mathcal{L}u=e^{-y\sqrt{|\xi|^2+\delta}}\widehat{g}$, 
and that we can split $u$ as 
\begin{eqnarray*}
u(x,y,t)
&=& \frac{1}{(2\pi)^{d}}\int_{\R^{d-1}}\int_{0}^\infty e^{i(y\eta+x\cdot\xi)}e^{-it(|\xi|^2+\eta^2)}
2\eta \widehat{g}(\xi,-\eta^2-|\xi|^2)d\eta\, d\xi\\
&&+\frac{1}{(2\pi)^{d}}\int_{\R^{d-1}}\int_{0}^\infty e^{-y\eta+ix\cdot\xi}e^{it(-|\xi|^2+\eta^2)}
2\eta\widehat{g}(\xi,-|\xi|^2+\eta^2)d\eta\, d\xi\\
&=& u_1+u_2.
\end{eqnarray*}
This splits the frequencies in two regions $\{\delta<-|\xi|^2\}:=\mathcal{R}_h$ and
$\{\delta >-|\xi|^2\}:=\mathcal{R}_e$. In the usual terminology of boundary value problems these are the hyperbolic
and elliptic regions (see \cite{Szeftel1} in the context of the Schr\"odinger equation).
According to Plancherel's formula, 
\begin{equation*}
 \|u_1(t=0)\|_{L^2}\sim \|\eta \widehat{g}(\xi,-\eta^2-\xi^2)\|_{L^2_{\xi,\eta}}\sim 
 \|\widehat{g}(\xi,\delta)\,||\xi|^2+\delta|^{1/4}\|_{L^2},
\end{equation*}
therefore the weight $\sqrt{||\xi|^2+\delta|}$ can not be modified in $\mathcal{R}_h$. \\
In $\mathcal{R}_e$, 
we set $J(\xi,\eta)=\sqrt{\eta/p(\xi,-|\xi|^2+\eta^2)}$, 
$\varphi(\xi,\eta)=2\widehat{g}(\xi,-|\xi|^2+|\eta|^2)/J(\xi,\eta)$. We remark that 
\eqref{questopt} is equivalent to $\sup J=+\infty$, moreover
\begin{equation*}
\int \varphi^2(\xi,-|\xi|^2+\eta^2)d\eta =\int_{\R^{d-1}}\int_{-|\xi|^2}^\infty|g(\xi,\delta)|^2p(\xi,\delta)
d\delta ,
\end{equation*}
so that $\widehat{g}$ is bounded in $L^2(pd\delta d\xi)$ if and only if $\varphi$ is in 
$L^2(d\eta d\xi)$.
Now without loss of generality we can assume that for any $\,(\xi,\eta),\ \varphi(\xi,\eta)\in \R^+$, and we bound 
\begin{eqnarray}\nonumber
\|u_2(\cdot,t)\|_{L^2_{x,y}}&\leq& \bigg\|\int_{0}^\infty e^{-y\eta}
2\eta\widehat{g}(\xi,-|\xi|^2+\eta^2)d\eta \bigg\|_{L^2_{\xi,y}}\\
\nonumber &=& \bigg\|\bigg(\int_{[0,\infty[^3} e^{-y(\eta_1+\eta_2)}\varphi(\xi,\eta_1)
\varphi(\xi,\eta_2)J(\xi,\eta_1)J(\xi,\eta_2)
d\eta_1 d\eta_2dy \bigg)^{1/2}\bigg\|_{L^2_\xi}\\
\label{hardytordu}&=&\bigg(\int_{\R^{d-1}}\int_{[0,\infty[^2}\frac{J(\xi,\eta_1)J(\xi,\eta_2)}{\eta_1+\eta_2}
\varphi(\xi,\eta_1)\varphi(\xi,\eta_2)
d\eta_1 d\eta_2 d\xi\bigg)^{1/2}
\end{eqnarray}
Using the decomposition $(\R^+)^2=\{\eta_1<\eta_2\}\cup \{\eta_2<\eta_1\}$,
we see that \eqref{hardytordu} is bounded by $\|\varphi\|_{L^2}^2$ if 
\begin{equation*}
T:\varphi\mapsto \frac{J(\xi,\eta_1)}{\eta_1}\int_0^{\eta_1}J(\xi,\eta_2)\varphi(\xi,\eta_2)d\eta_2
\text{ is bounded } L^2\to L^2.
\end{equation*}
% Without further assumptions it is easy to contruct weights $p$ that satisfy \eqref{questopt}.
% For $J(\eta)\in C(\R^+)\cap L^2(\R^+)$ such that $J$ is not bounded, we have 
% \begin{equation}\label{keyeq}
% \int_{\R^{d-1}}\int_0^\infty \varphi(\xi,\eta_1)\frac{J(\eta_1)}{\eta_1}\int_0^{\eta_1}
% \varphi(\xi,\eta_2)J(\eta_2)d\eta_2 d\eta_1 d\xi
% :=\int_{\R^{d-1}}\int_0^\infty \varphi(\xi,\eta) T(\varphi(\xi,\cdot))(\eta)d\eta d\xi.
% \end{equation}
% by Hardy and Cauchy-Schwarz's inequalities, 
% \begin{equation*}
% \|T\varphi(\xi,\cdot)\|_{L^2([0,1])}\lesssim \|\varphi(\xi,\cdot\|_{L^2},\ T\varphi(\xi,\eta)\lesssim 
% \frac{J(\eta)}{\eta}\|\varphi(\xi,\cdot)\|_{L^2}\in L^2([1,\infty[).
% \end{equation*}
% which implies the boundedness of $T$.
% Actually we can even choose $J\geq 1$ (so that $p\leq \sqrt{||\xi|^2+\delta|}$) by setting $J=1+r$, 
% $r\in C(\R^+)\cap L^2(\R^+)$ , $r$ positive unbounded, since then
% \begin{equation*}
% T\varphi=\frac{1}{\eta}\int_0^\eta \varphi(\nu)d\nu+\frac{r(\eta)}{\eta}\int_0^\eta \varphi(\nu)d\nu+
% \frac{1}{\eta}\int_0^\eta r(\nu)\varphi(\nu)d\nu+\frac{r(\eta)}{\eta}\int_0^\eta r(\nu)\varphi(\nu)d\nu,
% \end{equation*}
% and it is easily seen that each term defines an $L^2$ function.
% \\
Due to scaling invariances, it seems natural to add some homogeneity assumptions: if $u$ 
is a solution of the BVP 
with boundary data $g$, then $\lambda^{d/2}u(\lambda x,\lambda y,\lambda^2 t)$ is a solution with boundary 
data $g(\lambda x,\lambda^2t)$ and same $C_tL^2$ norm. The norm of the boundary data is scale invariant if 
\begin{equation*}
\int |\widehat{g}(\xi,\delta)|^2\frac{p(\lambda \xi,\lambda^2\delta)}{\lambda}d\xi d\delta=
\int |\widehat{g}(\xi,\delta)|^2p(\xi,\delta)d\xi d\delta,
\end{equation*}
which is true provided $p$ is anisotropically homogeneous: $p(\lambda\xi,\lambda^2\delta)=\lambda p(\xi,\delta)$. 
This is equivalent to the 0-homogeneity of 
$J(\xi,\eta)$. Somewhat surprisingly, even with these strong assumptions it is possible to construct $J$ satisfying 
\eqref{questopt}. 
\begin{prop}
There exists $p(\xi,\delta)$ such that \eqref{questopt} is true, moreover we can choose $p$ such that 
\begin{equation*}
\forall\,(\lambda,\xi,\delta)\in \R^{+*}\times \R^{d-1}\times \R,\ 
p(\lambda \xi,\lambda^2\delta)=\lambda p(\xi,\delta), \text{ and } p(\xi,\delta)\leq \sqrt{||\xi|^2+\delta|}.
\end{equation*}
\end{prop}
\begin{proof}
We keep the notations of the discussion above.
For simplicity, we assume $d=2$, and define :
\begin{equation*}
r(\xi,\eta)=
\left\{
\begin{array}{ll}
j\text{ if }2^j-2^{-j}\leq \frac{\eta}{\xi}\leq 2^j,\ j\in \N,\\
0\text{ else}.
\end{array}
\right.
 \end{equation*}
Obviously, $J:=1+r\geq 1$ is $0$-homogeneous and unbounded, thus 
$$p=\sqrt{\delta+\xi^2}/J^2(\xi,\sqrt{\delta+\xi^2})\leq \sqrt{\delta+\xi^2},\ \inf p=0\text{ and }
p(\lambda\xi,\lambda^2\delta) =\lambda p(\xi,\delta).$$
Developping in \eqref{hardytordu} 
$J(\xi,\eta_1)J(\xi,\eta_2)=1+r(\xi,\eta_1)+r(\xi,\eta_2)+r(\xi,\eta_1)r(\xi,\eta_2)$
it suffices to estimate each term separately. By symmetry, we can simply consider the integral 
over $\eta_1\geq \eta_2$. The term with $1$ is estimated thanks to Hardy's inequality, for the 
term with $r(\xi,\eta_2)$ we write
\begin{eqnarray*}
\int_0^\infty \frac{1}{\eta_1^2}\bigg(\int_0^{\eta_1}r(\xi,\eta_2)\varphi(\xi,\eta_2)d\eta_2\bigg)^2d\eta_1
 &\lesssim& \sum_{k=0}^\infty \int_{\xi 2^{k-1}}^{\xi 2^k}\frac{1}{\eta_1^2}
\bigg(\int_0^{2^k\xi}r\varphi(\xi,\eta_2)d\eta_2\bigg)^2d\eta_1\\
&\lesssim& \sum_{k=0}^\infty \frac{2^{-k}}{\xi}\bigg(\sum_{j=0}^k\int_{\xi(2^j-2^{-j})}^{\xi 2^j}j\varphi d\eta_2\bigg)^2
\\
&\lesssim&\sum_{k=0}^\infty \frac{2^{-k}}{\xi}\bigg(\sum_{j=0}^k\|\varphi(\xi,\cdot)\|_{L^2([\xi (2^j-2^{-j}),\xi 2^j])}
j\sqrt{2^{-j}\xi}\bigg)^2\\
&\lesssim &  \|\|\varphi(\xi,\cdot)\|_{L^2([\xi (2^j-2^{-j}),\xi 2^j])}\|_{l^2_j} \leq\|\varphi(\xi,\cdot)\|_{L^2_\eta}^2.
\end{eqnarray*}
Similarly for the term with $r(\xi,\eta_1)$
\begin{eqnarray*}
\int_0^\infty \frac{r^2(\xi,\eta_1)}{\eta_1^2}\bigg(\int_0^{\eta_1}\varphi(\xi,\eta_2)d\eta_2\bigg)^2d\eta_1
 &\lesssim& \sum_{k=0}^\infty \int_{\xi (2^{k}-2^{-k})}^{\xi 2^k}\frac{k^2}{\eta_1^2} 
\bigg(\int_0^{2^k\xi}\varphi(\xi,\eta_2)d\eta_2\bigg)^2d\eta_1\\
&\lesssim& \sum_{k=0}^\infty \frac{k^2 2^{-3k}}{\xi}\|\varphi(\xi,\cdot)\|_{L^2}^22^k\xi
\lesssim\|\varphi(\xi,\cdot)\|_{L^2_\eta}^2.
\end{eqnarray*}
The last term $r(\xi,\eta_1)r(\xi,\eta_2)$ is easier to estimate, we conclude by integration in $\xi$
\begin{equation*}
\int_{\R}\int_{(\R^+)^2} \frac{J(\xi,\eta_1)J(\xi,\eta_2)}{\eta_1+\eta_2}\varphi(\xi,\eta_1)
\varphi(\xi,\eta_2)d\eta_2 d\eta_1 d\xi \lesssim \|\varphi\|_{L^2(\R^{d-1}\times \R^+)}^2,
\end{equation*}
despite the fact that $J$ is larger than $1$ and unbounded.
\end{proof}
\begin{rmq}
Let us point out that the contribution of the elliptic region $\mathcal{R}_e$ to the solution corresponds 
to a superposition of so-called evanescent waves, that do not propagate like solutions of the Cauchy problem: 
for $(\delta,\xi)$ such that $\delta+|\xi|^2>0$, the wave $e^{-y\sqrt{||\xi|^2+\delta|}}e^{i(\delta t+x\cdot \xi)}$ is 
a solution of the Schr\"odinger equation on $\R^{d-1}\times \R^+$ remaining localized near the boundary. \\
As mentionned before, for frequencies that correspond to propagating waves, the weight $\sqrt{\delta+|\xi|^2}$ 
is optimal.
\end{rmq}
\paragraph{Acknowledgement} C.A. was  partially supported by the french ANR 
project BoND ANR-13-BS01-0009-01.

\bibliography{biblio}
 \bibliographystyle{plain}

\end{document}